\documentclass[pdflatex,sn-mathphys-num]{sn-jnl}% Math and Physical Sciences Numbered Reference Style
\usepackage{graphicx}%
\usepackage{multirow}%
\usepackage{amsmath,amssymb,amsfonts}%
\usepackage{amsthm}%
\usepackage{mathrsfs}%
\usepackage[title]{appendix}%
\usepackage{xcolor}%
\usepackage{textcomp}%
\usepackage{manyfoot}%
\usepackage{booktabs}%
\usepackage{algorithm}%
\usepackage{algorithmicx}%
\usepackage{algpseudocode}%
\usepackage{listings}%
\usepackage{float}%
\usepackage{tikz}%
\usepackage{url}%
\theoremstyle{thmstyleone}%
\newtheorem{theorem}{Theorem}[section]%  meant for continuous numbers
\newtheorem{proposition}[theorem]{Proposition}% 
\newtheorem{cor}[theorem]{Corollary}
\newtheorem{lem}[theorem]{Lemma}
\theoremstyle{thmstyletwo}%
\newtheorem{remark}{Remark}%

\theoremstyle{thmstylethree}%
\numberwithin{equation}{section}
\begin{document}

\title[Length-constrained and length-penalised curve diffusion flows of planar curves inside cones]{Length-constrained and length-penalised curve diffusion flows of planar curves inside cones}


\author*[1,2]{\fnm{Mashniah A.} \sur{Gazwani}}\email{magazwani@iau.edu.sa}
\equalcont{These authors contributed equally to this work.}
\author[2, 3]{\fnm{James A.} \sur{McCoy}}\email{ James.McCoy@newcastle.edu.au, James.McCoy@rmit.edu.au}
\equalcont{These authors contributed equally to this work.}
\email{mashniah.gazwani@uon.edu.au}
\affil*[1]{\orgdiv{Department of Mathematics}, \orgname{ College of Science, Imam Abdulrahman Bin Faisal University}, \orgaddress{\street{Prince Naif Road}, \city{Dammam}, \postcode{31441}, \state{Eastern Province}, \country{Kingdom of Saudi Arabia}}}

\affil[2]{\orgdiv{School of  Computer and Information Sciences}, \orgname{The University of Newcastle}, \orgaddress{\street{University Drive}, \city{Callaghan}, \postcode{2308}, \state{New South Wales}, \country{Australia}}}

\affil[3]{\orgdiv{Department of Mathematical and Geospatial Sciences}, \orgname{Royal Melbourne Institute of Technology}, \orgaddress{\street{124 La Trobe Street}, \city{Melbourne}, \postcode{3000}, \state{Victoria}, \country{Australia}}}


\abstract{We study families of smooth, embedded, regular planar curves
 $ \alpha : \left [-1,1  \right ]\times \left [0,T  \right )\to \mathbb{R}^{2}$ with generalised Neumann boundary conditions inside cones, satisfying three variants of the fourth-order nonlinear curve diffusion flow: (1) curve diffusion flow with a length penalisation, (2, 3) two forms of constrained curve diffusion flow with fixed length.  We prove in cases (2) and (3) for cone angle less than $\pi$, if the initial curve has small oscillation of curvature and the initial curve is sufficiently far from the cone tip, then the solution exists for all time and converges exponentially in the $C^\infty$-topology to a circular arc with the same length as the initial curve.  In case (1), a similar result holds under suitable rescaling.  In all cases, the limiting arc is centred at the cone tip. }

\keywords{Curve diffusion flow,  fourth-order geometric evolution equation, fourth-order parabolic equation, Neumann boundary condition}



\maketitle

\section{Introduction}\label{sec1}
The constrained curve diffusion flows considered in this article are fourth-order nonlinear geometric evolution equations based upon the original curve diffusion flow, with a lower-order term added to the normal speed which is either a length penalisation or has the effect of preserving length under the evolution.  The original curve diffusion flow of closed plane curves was considered in \cite{Wheeler2013}.  Together with Wheeler and Wu, the second author considered a length-preserving curve diffusion flow for closed curves in \cite{MWY}.  In both cases it was shown that initial curves suitably geometrically-close to a circle converge asymptotically in infinite time to a circle.

In a recent paper \cite{GM24}, the authors considered the original curve diffusion flow for curves with generalised Neumann boundary conditions (see Section 2 for the precise definition) inside cones.  For initial curves of sufficiently small oscillation of curvature, it was shown that, provided neither end of the curve reaches the cone tip, a solution exists for all time and converges exponentially to a unique circular arc, centred at the cone tip, bounding together with the sides of the cone the same area as the initial curve.  It is worth noting that, as compared with the case of closed curves in \cite{Wheeler2013}, where smallness conditions on both the oscillation of curvature and isoperimetric ratio of the initial curve were required, in \cite{GM24}, and in this article, only a small oscillation of curvature condition is required, since the problematic term in the evolution of the oscillation of curvature can be handled differently when the cone angle is less than $\pi$.

In this article we prove similar behaviour for the length-penalised and length-constrained curvature flows.  It is easy to see by examining the behaviour of circular arcs that the length-penalised flow tends to drive solutions to the cone tip in finite time.  While it is common to add a length penalisation term to the elastic energy to avoid solutions expanding to infinity (see, e.g. \cite{GM25} for the setting in cones), to our knowledge the addition of a lower order penalisation term to the curve diffusion flow is new.  This may be explained by the fact that the curve diffusion flow naturally preserves the enclosed area for closed embedded curves, or a natural generalisation of this in some other situations.  We suspect however that our approach of adding a penalisation term may be of interest in the future in other situations where it is less clear what the analogue of enclosed area should be.

The structure of this article is as follows.  In Section \ref{S:prelim} we outline the set-up of our flow problem, the boundary conditions and particular speeds of interest.  We define the key quantities and state some fundamental inequalities to be used in the subsequent analysis.  We conclude the section with some important properties that hold for each of the flows under consideration.  In the subsequent three sections we analyse the three flows in turn, highlighting the key differences in each case. 

We remark that very recently we learned of a new result by Hiroi and Okabe \cite{HO} concerning the curve diffusion flow between skew lines, which is a set-up more general that the cone presented here.  We believe that looking at the associated flows in this more general setting is an interesting question for future work.

\section{Preliminaries} \label{S:prelim}
 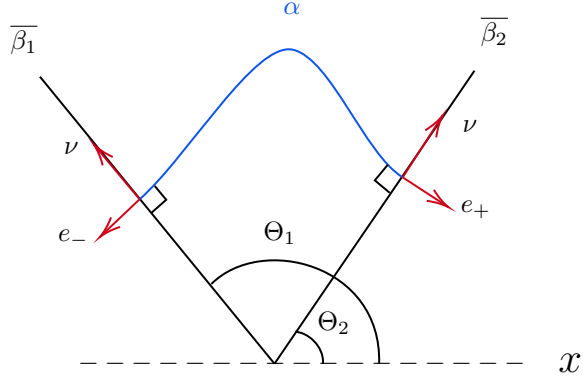
\begin{figure}[H] 
  \centering

\tikzset{every picture/.style={line width=0.75pt}} %set default line width to 0.75pt        

\begin{tikzpicture}[x=0.75pt,y=0.75pt,yscale=-1,xscale=1]
%uncomment if require: \path (0,310); %set diagram left start at 0, and has height of 310

%Straight Lines [id:da330280283577162] 
\draw    (82.33,86) -- (200,230) ;
%Straight Lines [id:da732524810594248] 
\draw    (300.33,83) -- (200,230) ;
%Shape: Right Angle [id:dp9705069174393858] 
\draw   (139.91,140.82) -- (145.69,147.75) -- (138.11,154.43) ;
%Curve Lines [id:da9304722164649084] 
\draw [color={rgb, 255:red, 12; green, 84; blue, 234 }  ,draw opacity=1 ]   (132.33,147.5) .. controls (156.33,126.5) and (190.33,68) .. (209.33,72.5) .. controls (228.33,77) and (245.33,125) .. (264.33,136.5) ;
%Shape: Right Angle [id:dp37785249528441867] 
\draw   (257.84,144.27) -- (250.33,138) -- (256.82,130.23) ;
%Curve Lines [id:da779885922178895] 
\draw    (168.33,190) .. controls (193.33,165) and (254,180) .. (252.5,230) ;
%Straight Lines [id:da5942750271612008] 
\draw [color={rgb, 255:red, 208; green, 2; blue, 27 }  ,draw opacity=1 ]   (132.33,147.5) -- (113.78,165.12) ;
\draw [shift={(112.33,166.5)}, rotate = 316.47] [color={rgb, 255:red, 208; green, 2; blue, 27 }  ,draw opacity=1 ][line width=0.75]    (10.93,-3.29) .. controls (6.95,-1.4) and (3.31,-0.3) .. (0,0) .. controls (3.31,0.3) and (6.95,1.4) .. (10.93,3.29)   ;
%Straight Lines [id:da7173654921964416] 
\draw [color={rgb, 255:red, 208; green, 2; blue, 27 }  ,draw opacity=1 ]   (264.33,136.5) -- (285.66,150.41) ;
\draw [shift={(287.33,151.5)}, rotate = 213.11] [color={rgb, 255:red, 208; green, 2; blue, 27 }  ,draw opacity=1 ][line width=0.75]    (10.93,-3.29) .. controls (6.95,-1.4) and (3.31,-0.3) .. (0,0) .. controls (3.31,0.3) and (6.95,1.4) .. (10.93,3.29)   ;
%Straight Lines [id:da5338141504602664] 
\draw [color={rgb, 255:red, 208; green, 2; blue, 27 }  ,draw opacity=1 ]   (132.33,147.5) -- (110.62,121.53) ;
\draw [shift={(109.33,120)}, rotate = 50.09] [color={rgb, 255:red, 208; green, 2; blue, 27 }  ,draw opacity=1 ][line width=0.75]    (10.93,-3.29) .. controls (6.95,-1.4) and (3.31,-0.3) .. (0,0) .. controls (3.31,0.3) and (6.95,1.4) .. (10.93,3.29)   ;
%Straight Lines [id:da6805372914439409] 
\draw [color={rgb, 255:red, 208; green, 2; blue, 27 }  ,draw opacity=1 ]   (264.33,136.5) -- (283.25,107.18) ;
\draw [shift={(284.33,105.5)}, rotate = 122.83] [color={rgb, 255:red, 208; green, 2; blue, 27 }  ,draw opacity=1 ][line width=0.75]    (10.93,-3.29) .. controls (6.95,-1.4) and (3.31,-0.3) .. (0,0) .. controls (3.31,0.3) and (6.95,1.4) .. (10.93,3.29)   ;
%Curve Lines [id:da2095801919797763] 
\draw    (210.33,214) .. controls (216.33,214) and (224.92,221) .. (224.33,230) ;

% Text Node
\draw (203,47.4) node [anchor=north west][inner sep=0.75pt]  [color={rgb, 255:red, 15; green, 87; blue, 233 }  ,opacity=1 ]  {$\alpha $};
% Text Node
\draw (302,54.4) node [anchor=north west][inner sep=0.75pt]    {$\overline{\beta _{2}}$};
% Text Node
\draw (66,58.4) node [anchor=north west][inner sep=0.75pt]    {$\overline{\beta _{1}}$};
% Text Node
\draw (193,157.4) node [anchor=north west][inner sep=0.75pt]    {$\Theta _{1}$};
% Text Node
\draw (220,202.4) node [anchor=north west][inner sep=0.75pt]    {$\Theta _{2}$};
% Text Node
\draw (92.73,117.62) node [anchor=north west][inner sep=0.75pt]  [rotate=-357.95]  {$\nu $};
% Text Node
\draw (292.73,106.62) node [anchor=north west][inner sep=0.75pt]  [rotate=-357.95]  {$\nu $};
% Text Node
\draw (90,161.4) node [anchor=north west][inner sep=0.75pt]    {$e_{\mathbf{-}}$};
% Text Node
\draw (292,146.4) node [anchor=north west][inner sep=0.75pt]    {$e_{\mathbf{+}}$};
% Text Node
\draw (100,223.72) node [anchor=north west][inner sep=0.75pt]    {$-----------------$};
% Text Node
\draw (341,223.4) node [anchor=north west][inner sep=0.75pt]  [font=\Large]  {$x$};

\end{tikzpicture}

\caption[The set-up.]{Curve with generalised Neumann boundary conditions in the cone.\footnotemark}
 \label{fig1}
\end{figure} 
 \footnotetext{\textbf{Fig.\ref{fig1}} was designed with the help of Mathcha Editor \url{https://www.mathcha.io/editor}.}

Given an initial smooth, immersed, regular open planar curve with boundary $\alpha_0: \left[ -1, 1\right] \rightarrow \mathbb{R}^2$, we consider the family of curves $\alpha\left( \cdot, t\right)$ moving with normal velocity $F$: \begin{equation}\label{f}
    \frac{\partial }{\partial t}\alpha  =-\,F[k]\,\nu,
\end{equation}
where $F[k]$ is the normal speed of the curve and $\alpha( \cdot, 0) = \alpha_0$.
Denoting the spatial parameter by $u$ and tangent vector by $\alpha_u$, above $\nu$ is the normalised normal vector to $\alpha$ obtained by rotating $\alpha_u$ counter clockwise by $\frac{\pi}{2}$.  Furthermore, $k=-\left<\kappa, \nu  \right>=-\left<  \alpha_{ss}, \nu \right >$  is the scalar curvature of $\alpha$, where $\left<  \cdot, \cdot \right >$ is the usual Euclidean inner product in $\mathbb{R}^2$ and $s$ is the arc-length parameter of curve $\alpha$, defined via
\begin{equation} \label{s}
  s\left( u, t\right) = \int_{-1}^{u} \left| \alpha_u \left( \tilde u, t\right) \right| d\tilde u \mbox{.}
\end{equation}
The sign in \eqref{f} is chosen such that the equation is parabolic in the generalised sense.  Let us further denote by $\tau =\displaystyle{\frac{\alpha_{u} }{\left|\alpha _{u} \right|}}=\alpha _{s}$ the unit tangent vector field along $\alpha$.

The three speeds $F$ of interest in this article are
\begin{enumerate}
    \item[\textnormal{(1)}] $F= -k_{ss} + \lambda\, k$ for any constant $\lambda>0$  (the `length-penalised curve diffusion flow');
    \item[\textnormal{(2)}] $F= -k_{ss} +\lambda_1(t) $; 
    \item[\textnormal{(3)}] $F= -k_{ss} +\lambda_2(t)\,k $\\
    for specific functions $\lambda_1(t)$ and $\lambda_2(t)$ chosen such that the length of the evolving curve is fixed (`length-constrained curve diffusion flows')
\end{enumerate}

The boundary conditions for our evolving planar curve comprise the classical Neumann condition and a no-curvature-flux condition.  Fix $0\leq \Theta_2 < \Theta_1 < 2\pi$ and let $\beta_{i} :\left [ 0,\infty  \right )\to \mathbb{R}^{2}~ (i=1,2)$, such that $ \beta_{i}\left ( \rho  \right ) = (\rho \cos \Theta_i, \rho \sin \Theta_i),$ and denote the image sets by 
$$\bar{\beta}_{i} :\left\{\left ( x,y \right )\in \mathbb{R}^{2}: \left ( x,y \right )=\left ( \rho \cos\Theta _{i}, \rho\sin \Theta _{i}\right ), \rho > 0 \right\}.$$
The set $\bar{\beta}_{1} \cup \bar{\beta}_{2}$ forms the boundary for our evolving curve, namely a cone in $\mathbb{R}^2$ with the tip at the origin and cone angle $\Theta_1 - \Theta_2 < 2\pi$.  We remark that in our recent work \cite{GM24} on the curve diffusion flow, we required the cone angle to be strictly less than $\pi$ in order to have a preserved smallness condition on the oscillation of curvature under the flow.  The same is true in this article.  It is an interesting question whether there is a different smallness quantity or another approach that could facilitate a result in case of cone article larger than $\pi$.  Indeed, for the free elastic flow \cite{GM25} we obtained a convergence result (under rescaling) with a different smallness condition but restriction on the cone angle at all.

The cone boundaries and initial curve $\alpha_0$ enclose a finite area that we call the `interior'.  This region is assumed to be contained within the set 
$$\left\{ \left( x, y\right) \in \mathbb{R}^2: \left( x, y\right) = \left( \rho \cos \Theta, \rho \sin \Theta\right), \rho >0, \Theta_2 < \Theta < \Theta_1 \right\} \mbox{,}$$
 that is, `inside' the cone and such that $\alpha_{0} \left ( -1 \right )\in \bar{\beta} _{1}$ and $ \alpha_{0} \left ( 1 \right )\in \bar{\beta} _{2}$.  Moreover, $\alpha_0$ is assumed to meet the boundary at each endpoint perpendicularly (the classical Neumann boundary condition) and be such that the curvature $k_0$ of $\alpha_0$, computed using one-sided derivatives, satisfies
$$\left( k_0\right)_s\left( -1\right) = \left( k_0\right)_s\left( 1 \right) = 0 \mbox{,}$$
the so-called \emph{no-curvature-flux condition} on the boundary.  The cone tip itself is not included in the specification of the boundary condition, as the normal to the boundary there is not well defined.  We now let the curve evolve under \eqref{f} with \emph{generalised Neumann boundary conditions} as long as this is a well-posed problem.  This means that under the evolution, the classical Neumann boundary condition and no-curvature-flux condition continue to hold.  In view of short-time existence result, the solution will not immediately jump to the cone tip.  It is indeed a key feature of our analysis that we have to avoid the ends of the evolving curve reaching the cone tip, unless they do so in a limiting sense towards the maximal time $T$ of existence of the solution.  In our setting, for the length constrained flows we will have $T=\infty$ under the assumption neither end reaches the cone tip, while for the length penalised flows ($\lambda>0$ constant) we have $T<\infty$.  Our smallness conditions facilitate smooth exponential convergence to the limiting circular arc (under rescaling for length in the case $\lambda>0$ constant) and thus, in particular, the flow speed decays exponentially.  This in turn means the distance travelled by any point on the evolving curve, including the boundary points, is bounded and so if the boundary points are sufficiently far from the cone tip, then they will not reach the cone tip under the evolution.

Denoting by $e_-$ and $e_+$ unit vectors perpendicular to $\bar{\beta} _{1}$ and $\bar{\beta} _{2}$, respectively, 
the boundary conditions ensure that the ends of the evolving curve continue to meet either side of the cone as long as the solution to \eqref{f} exists and
 \begin{equation} \label{E:NBC1}
   \left< \nu\left( -1, t\right), e_- \right> = \left< \nu\left( 1, t\right), e_+ \right> = 0
   \end{equation}
   and
   \begin{equation} \label{E:NBC2}
   k_s\left( \pm 1, t\right) = 0 \mbox{.}
   \end{equation}
   By slight abuse of notation, we will denote quantities associated with the evolving curve using the same symbols, whether they be functions of $u$ and $t$, or $s$ and $t$, as is customary, and should not lead to any confusion.  When we evaluate at endpoints, we will denote the spatial argument using $\pm 1$, but spatial derivatives will be with respect to $s$ and interpreted in the appropriate one-sided sense.  We denote by $k_{s^{n}}$ the $n$-th iterated derivative of $k$ with respect to arc-length and write $k_{s^{n}}^{2}$ for $\left ( k_{s^{n}} \right )^{2}$.

 Throughout the article, we use $L$ to denote the length of $\alpha\left( \cdot, t\right)$ and write
\begin{equation}\label{eql}
    L\left( t\right) = L[\alpha\left( \cdot, t\right) ]=\int_{\alpha }\, ds.
\end{equation}
All integrals will be over the curve $\alpha$ unless otherwise indicated.  We also have the area $A$ of the region bounded by the curve and cone
\begin{equation} \label{E:area}
  A\left( t\right) = A\left[\alpha\left( \cdot, t\right) \right] = +\frac{1}{2} \int_{\alpha } \left< \alpha, \nu \right> ds \mbox{,}
\end{equation}
where the sign is $+$ due to our choice of outer unit normal.  

The average curvature of $\alpha\left( \cdot, t\right)$  is defined as\begin{equation}\label{eqkk}
 \bar{k}[\alpha]:=\displaystyle{\frac{1}{L}}\int_{\alpha } k\,ds.   
\end{equation}
Using this, we define the scale-invariant oscillation of curvature
\begin{equation} \label{E:Kosc}
  K_{\mbox{osc}} := L \int_\alpha \left( k - \overline{k}\right)^2 ds \mbox{.}
\end{equation}
Additionally, the rotation number of $\alpha\left( \cdot, t\right)$ is
\begin{equation} \label{E:omega}
  \omega \left [ \alpha  \right ] :=\frac{1}{2\pi }\int_{\alpha } k\,ds.
  \end{equation}

This may be thought of as the general definition of rotation number
for any curve for which the integral makes sense.  As the rotation number gives the net turning of the tangent vector to $\alpha$ we observe that in our setting 
$$\omega = \frac{\Theta_1 - \Theta_2}{2\pi} \mbox{.}$$

To prove our main theorems, we require some well-known fundamental tools.  First are the standard Poincaré-Sobolev-Wirtinger [PSW] inequalities for curves with boundaries.

\begin{proposition} \label{psw}
Let $L>0$. Let $g:[0,L]\rightarrow \mathbb{R}$ be an absolutely continuous function with $\int_{0}^{L}g(x) \,dx=0.$ Then 
\begin{equation}\label{eqpsw}
    \int_{0}^{L}g^{2}(x)\, dx\leq \frac{L^{2}}{\pi ^{2}}\int_{0}^{L}\left( g_{x} \right)^{2}(x)\, dx.
\end{equation}
Similarly, if $g:[0,L]\rightarrow \mathbb{R}$ is absolutely continuous and $g(0)=g(L)=0$ then 
\begin{equation*}
\int_{0}^{L}g^{2}(x) \,dx\leq \frac{L^{2}}{\pi ^{2}}\int_{0}^{L} \left( g_{x} \right)^{2}(x)\,dx.
\end{equation*}
\end{proposition}
\begin{proposition} \label{eq266}
If $g:[0,L]\rightarrow \mathbb{R}$ is absolutely continuous and $g(0)=g(L)=0$ then 
\begin{equation}\label{pp}
\left \| g \right \|_{\infty }^{2}\leq \frac{L}{\pi }\int_{0}^{L}\left( g_{x} \right)^2 (x)\, dx.
\end{equation}
Similarly, if  $g:[0,L]\rightarrow \mathbb{R}$ is absolutely continuous and $\int_{0}^{L}g \,dx=0$ then
\begin{equation}\label{eqppsw}
\left \| g \right \|_{\infty }^{2}\leq \frac{2L}{\pi }\int_{0}^{L}\left( g_{x} \right)^2 (x)\, dx.
\end{equation}
\end{proposition}
It is convenient, as in \cite{DKS02} and subsequent work, to use the following scale-invariant norms: 
 $$\left\| k \right\|_{\ell, p} := \sum_{i=0}^\ell \left\|  k_{s^{i}} \right\|_p$$
 where
 $$\left\|  k_{s^{i}} \right\|_p = L^{i+1-\frac{1}{p}} \left( \int_\alpha \left|  k_{s^{i}} \right|^p ds\right)^{\frac{1}{p}} \mbox{.}$$
 
 The following interpolation inequality for curves with boundary is \cite{DP14} [Lemma 4.3].  The analogous inequalities for closed curves appeared earlier in \cite{DKS02}.  We use the standard notation $P_n^m(k)$ to denote a linear combination of terms each of which contain $n$ factors of $k$ with a total of $m$ derivatives with respect to $s$.
 
\begin{proposition}\label{p}
  Let $\alpha:\left [ -1, 1 \right ]\to \mathbb{R}^{2}$  be a smooth curve with boundary. Then for any term $P_{n}^{m}\left ( k \right )$ with $n\geq 2$  that contains derivatives of $k$ of order at most $\ell-1$,
  $$\int_\alpha \left|P_{n}^{m}\left ( k \right ) \right| ds\leq c\, L^{1-m-n}\left\| k\right\|_{0,2}^{n-p}\left\| k\right\|_{\ell,2}^{p},$$
  where $p=\displaystyle{\frac{1}{\ell}\left ( m+\frac{1}{2} n-1\right )}$ and $c=\left ( \ell,m,n \right ) $. Furthermore, if $m+\frac{n}{2}< 2\ell+1$ then $p<2$ and for any $\varepsilon>0,$
  $$\int_\alpha \left|P_{n}^{m}\left ( k \right ) \right| ds\leq \varepsilon \int_\alpha  k_{s^{\ell}}^{2}ds+c\,\varepsilon^{\frac{-p}{2-p}}\left( \int_\alpha k^{2}ds  \right)^{\frac{n-p}{2-p}}+c \left(\int_\alpha k^{2}ds  \right)^{m+n-1}.$$ 
\end{proposition}

Next, we record some evolution equations for general normal flow speeds $F$.  These are easily derived, for example, in \cite{Wheeler2013}.
\begin{lem}\label{eq of general F}
Under the flows \eqref{f}, we have the following evolution equations:
\begin{enumerate}
 \renewcommand{\labelenumi}{(\roman{enumi})}
 \item $\displaystyle{\frac{\partial }{\partial t} \, ds =  -F\, k\, ds}$
 \item $ \displaystyle{\frac{\mathrm{d} }{\mathrm{d} t}L = - \int_{\alpha } k\, F \, ds;}$
 \item $\displaystyle{ \frac{\mathrm{d} }{\mathrm{d} t} A = - \int_\alpha F\, ds;}$
 \item  For each $\ell \in \mathbb{N} \cup \left\{ 0 \right\}$, we have $\displaystyle{\frac{\partial}{\partial t} k_{s^\ell} = F_{s^{\ell + 2}} + \sum_{j=0}^\ell \left( k\, k_{s^{\ell-j}} F\right)_{s^j} \mbox{;}}$
 \item $\displaystyle{\frac{\mathrm{d} }{\mathrm{d}t} \int_{\alpha } k^2 \, ds =  2 \int_{\alpha } \left( k_{ss} + \frac{1}{2} k^3 \right) F \, ds;}$
 \item  $\displaystyle{\frac{\mathrm{d} }{\mathrm{d} t} \int_{\alpha } \left ( k-\bar{k} \right )^2  ds=2 \int _{\alpha }\left [ k_{ss}+\frac{1}{2}k\left ( k^{2}-\bar{k}^{2} \right )\right ]F\,ds };$
 \item $\displaystyle{\frac{\mathrm{d} }{\mathrm{d} t} \int_{\alpha } k_{s}^2\,  ds = 2 \int_{\alpha } \left( - k_{s^4} - k^2 k_{ss} + \frac{1}{2} k\, k_s^2 \right) F \, ds;}$
 \end{enumerate}
\end{lem}

While a smooth solution to the flow \eqref{f} exists, the rotation number $\omega$ is constant.  We state this fact as a Corollary.  The proof is a straightforward calculation, as in \cite{GM24} for example.

\begin{cor} \label{T:omega}
 Under the flow \eqref{f}, with speed (1), (2) or (3) and generalised Neumann boundary conditions \eqref{E:NBC1} and \eqref{E:NBC2} we have
 \begin{equation}
     \frac{\mathrm{d} }{\mathrm{d} t}\int _{\alpha }k\, ds=0.
 \end{equation}
% on the boundary.
\end{cor}

Similarly as in \cite{WW}, all odd curvature derivatives are equal to zero on the boundary, as long as the solution to \eqref{f} exists for each of the speed functions of interest.  As noted earlier, we use this fact when we `integrate by parts', that is, apply the Divergence Theorem to various functions involving curvature and its derivatives on our curve with boundary.

\begin{lem}\label{T:BCs}
Under the flows \eqref{f}, with the speeds (1), (2) or (3) and generalised Neumann boundary conditions \eqref{E:NBC1} and \eqref{E:NBC2}, we have for each $\ell \in \mathbb{N}$ 
  $$k_{s^{2\ell+1}}\left( \pm 1, t\right)=0 \mbox{.}$$
\end{lem}

As used in \cite{MWY}, the following inequality is needed for estimating terms with $\lambda_1(t)$ and $\lambda_2(t)$.  The proof is by induction and uses the H\"{o}lder inequality and integration by parts (ie the Divergence Theorem).  Whereas the setting in \cite{MWY} was closed curves, the boundary conditions \eqref{E:NBC2} (see also Lemma \ref{T:BCs}) ensure the boundary terms arising from integration by parts in the proof of the following in this case are equal to zero, so the proof is essentially the same.

\begin{lem}\label{k8}For each $n \in \mathbb{N}$
$$\int_{\alpha } k^2_{ s^{n-1}} \, ds \leq \left( \int_{\alpha } k^2 \, ds \right)^{\frac{1}{n}} \left( \int_{\alpha } k^2_{s^n} \, ds \right)^{\frac{n-1}{n}}.$$   
\end{lem}
\begin{proof}
We prove this by induction.
$$\int_{\alpha } k_{s^{n-1}}^2 \, ds \leq \left( \int_{\alpha } k^2 \, ds \right)^{\frac{1}{n}} \left( \int_{\alpha } k_{s^n}^2 \, ds \right)^{\frac{n-1}{n}}.$$ The statement is trivial for $n = 1$. Therefore, assume that
\begin{equation}\label{lambda12}
\int_{\alpha } k_{s^{i-1}}^2 \, ds \leq \left( \int_{\alpha } k^2 \, ds \right)^{\frac{1}{i}} \left( \int_{\alpha } k_{s^i}^2 \, ds \right)^{\frac{i-1}{i}}
\end{equation}
and we use this to prove that 
$$\int_{\alpha } k_{s^i}^2 \, ds \leq \left( \int_{\alpha } k^2 \, ds \right)^{\frac{1}{i+1}} \left( \int_{\alpha } k_{s^{i+1}}^2 \, ds \right)^{\frac{i}{i+1}}.$$
By integration by parts, Lemma \ref{T:BCs}  and the H\"{o}lder inequality, we obtain
$$\int_{\alpha} k_{s^i}^2 \, ds =  k_{s^{i}} k_{s^{i-1}}\bigg{|}_{\partial \alpha}  - \int_{\alpha} k_{s^{i-1}} k_{s^{i+1}} \, ds 
\leq \left( \int_{\alpha} k_{s^{i-1}}^2 \, ds \right)^{\frac{1}{2}}
\left( \int_{\alpha} k_{s^{i+1}}^2 \, ds \right)^{\frac{1}{2}},$$ where the boundary term above contains an odd derivative, i.e., if $i$ is odd, then $k_{s^i}=0$ on the boundary, while if $i$ is even, then $i-1$ is odd, then $k_{s^{i-1}}=0$ on the boundary.
By inserting on the right-hand side of \eqref{lambda12}, we have   
$$\int_{\alpha} k_{s^i}^2 \, ds 
\leq \left( \int_{\alpha} k^2 \, ds \right)^{\frac{1}{2i}} 
\left( \int_{\alpha} k_{s^i}^2 \, ds \right)^{\frac{i-1}{2i}} 
\left( \int_{\alpha} k_{s^{i+1}}^2 \, ds \right)^{\frac{1}{2}} ;$$
which implies that $$\left( \int_{\alpha} k_{s^i}^2 \, ds \right)^{\frac{i+1}{2i}} 
\leq \left( \int_{\alpha} k^2 \, ds \right)^{\frac{1}{2i}} 
 \left( \int_{\alpha} k_{s^{i+1}}^2 \, ds \right)^{\frac{1}{2}},$$
 thus, $$\int_{\alpha} k_{s^i}^2 \, ds 
\leq \left( \int_{\alpha} k^2 \, ds \right)^{\frac{1}{i+1}} 
\left( \int_{\alpha} k_{s^{i+1}}^2 \, ds \right)^{\frac{i}{i+1}}$$
which completes the proof.    
\end{proof}
%\textcolor{red}{Suggest including a sketch of the proof, that shows how the zero boundary for odd curvature derivatives comes in, and what is done in the case where even derivatives come up}

Short time existence of solutions in this setting is well-known.  We refer the reader to \cite{Wu} for example, where a general sketch for higher-order parabolic flows of curves with boundary conditions is provided. In cases (2) and (3), the global constraint term may be handled by a fixed point argument.  By continuity of solutions, neither end can immediately jump to the cone tip.

As is often the case for parabolic geometric flows, the finiteness of the maximal existence time $T$ of solutions under our flows is characterised by blow up of $\int_{\alpha} k^2 ds$.  We have the same result here, under the additional assumption that neither end of the evolving curve reaches the cone tip.

\begin{theorem}\label{blowup}
 Let $\alpha :\left [ -1, 1 \right ]\times \left [ 0,  T\right )\to \mathbb{R}^{2}$ be a maximal solution of \eqref{f} with speed (1), (2) or (3)  and generalised Neumann boundary conditions \eqref{E:NBC1}  and \eqref{E:NBC2}. If $T<\infty$, neither end of $\alpha$ has reached the cone tip and $L\left[ \alpha_t\right] \geq \underline{L}>0$  then there is a constant $c>0$ such that
 \begin{equation}
\int _{\alpha }k^{2}\,ds\geq c\left ( T-t \right )^{-\frac{1}{4}}     \mbox{.}
 \end{equation} 
\end{theorem}

Notice that from the H\"{o}lder inequality
\begin{equation} \label{E:Holder}
    \int_{\alpha} k^2 ds \geq \frac{4\pi^2 \omega^2}{L} \mbox{,}
\end{equation}
so $L \searrow 0$ is one way case of blow up of $\int k^2 ds$.  In this case, in view of the boundary conditions, the solution shrinks to the cone tip at time $T$.  

On the other hand, while $\int_{\alpha} k^2 ds$ remains bounded and $L\left[ \alpha_t\right] \geq \underline{L} >0$, one has control of all curvature derivatives in $L^2$.  The following result is standard and proved in the case of curves with boundary similarly as in \cite{DP14}.  We will highlight the key differences in the proof of Proposition \ref{bd} for each of our flows in the subsequent sections.  We have added in the specific requirement that $L \geq \underline{L}>0$ on $\left[ 0, T\right)$, which is important in the length penalised case.

\begin{proposition}\label{bd}
Let $\alpha: \left[ -1, 1\right]\times  \left[ 0, T \right )\to \mathbb{R}^{2}$ is a solution to \eqref{f} with speed (1), (2) or (3), compatible with the generalised Neumann boundary conditions \eqref{E:NBC1} and \eqref{E:NBC2} such that $L\left[ \alpha_t\right] \geq \underline{L}>0$. If  $\int_{\alpha } k^2 ds$  is bounded by a constant $C_{0}>0$\,, for all $t\in \left[ 0, T\right)$, %then $T=\infty$ 
then for each $\ell \in \mathbb{N}\cup \left\{0 \right\},$ there exists a constant $c_{\ell}> 0 $ such that %$$\int _{\alpha }k_{s^\ell}^{2}\,ds\leq C_{\ell}^{2}$$ 
$$\int _{\alpha }k_{s^{\ell}}^{2}\,ds\bigg|_t\leq \int _{\alpha }k_{s^{\ell}}^{2}\,ds\bigg|_0+c_\ell\,C_{0}^{2\ell+5}\,t\,\mbox{.}$$   
\end{proposition}

\begin{remark}
 As shown above, while $\int_{\alpha} k^2 ds$ is uniformly bounded, $\int_{\alpha} k_{s^\ell}^2\, ds$ may grow linearly in $t$.  In our circumstances this will not be an issue as exponential decay `beats' this linear growth in $t$ in our interpolation arguments.  In the length penalised case, such arguments will be done using the rescaled time variable.
\end{remark}

With the above preliminaries in place, we now discuss in the following sections the three cases of our flow.

\section{Curve diffusion flow with length penalty, \texorpdfstring{$\lambda > 0$ }{} } \label{S:lp}
In this section, we consider the flow equation
\begin{equation}\label{eq116}
\frac{\partial \alpha }{\partial t}=(  k_{ss}-\lambda \, k)~\nu \mbox{.}
\end{equation}

We remark that by using a time rescaling this flow can also be thought of as a higher order regularisation of the curve shortening flow, adding a small $\frac{1}{\lambda} k_{ss}$ term to avoid situations where the second order flow may be backward parabolic.  Such an approach is used in \cite{mikula2010simple}, for example.  More commonly, a small amount of the Euler-Lagrange operator associated to the elastic energy, $k_{ss} + \frac{1}{2} k^3$ is added \cite{HV06}.  In this section, however, we have $\lambda >0$ which, without the $k_{ss}$ term corresponds to parabolic curve shortening flow. 

Observe that a circular arc of radius $r$ centred at the cone tip evolves under \eqref{eq116} according to 
$$\displaystyle{\frac{\mathrm{d} }{\mathrm{d} t}r} = -\frac{\lambda}{r} \mbox{.}$$
In particular, an arc of radius $r_0$ at time $t=0$ evolves with
$$r(t) = \sqrt{r_0^2 - 2\lambda\, t} \mbox{.}$$
As $\lambda>0$, the arc shrinks self-similarly to the cone tip by time $T=\frac{r_0^2}{2\lambda}$.  The smaller the $\lambda$, the slower the circle shrinks and the longer the time taken to shrink to the cone tip.  In the limit $\lambda \searrow 0$ we have the curve diffusion flow, for which circles are stationary.  (If $\lambda <0$ then the solution for the evolving arc exists for all time and the arc expands indefinitely.  As mentioned before, such a situation may be worth considering in future work, where the $k_{ss}$ term in the speed serves as a regularisation.)

It is reasonable to assume that an initial curve geometrically close to a circular arc will evolve under \eqref{eq116} in a similar way as the self-similar shrinking circular arc.  In this section we show this is indeed the case, provided the initial curve has small enough oscillation of curvature.  Since $\lambda > 0$, the solution to \eqref{eq116} exists for a finite time and, provided neither end reaches the cone tip while the length remains nonzero, then the solution converges to the cone tip.  Under rescaling, the solution approaches a self-similar shrinking arc asymptotically. In the case $\lambda <0$, the solution exists for all time and  provided neither end reaches the cone tip while the length remains nonzero, then the solution converges to an expanding circular arc after rescaling.  %\textcolor{red}{clarify} \textcolor{green}{Done}

Under the flow \eqref{eq116} we have the following evolution of various geometric quantities.  These equations are easy to derive using similar calculations as in \cite{WW}, for example.

\begin{lem} \label{T:evlneqns1}
Under the flow \eqref{eq116} we have the following evolution equations for various pointwise geometric quantities.
\begin{itemize}
  \item[\textnormal{(i)}] $\displaystyle{\frac{\partial}{\partial t} ds} = (  k_{ss}-\lambda~ k)\, k\, ds$;
  \item [(ii)]$\displaystyle{\frac{\partial }{\partial t}k}=-k_{s^{4}}-k^{2}k_{ss}+\lambda k_{ss}+\lambda k^{3};$ 
\item [(iii)]$\displaystyle{\frac{\partial }{\partial t}}k_{s}=-k_{s^{5}}-3kk_{s}k_{ss}-k^{2}k_{s^{3}}+4\lambda k^{2}k_{s}+\lambda k_{s^{3}};$
  \item[\textnormal{(iv)}] For each $\ell= 0, 1, 2, \ldots$,
  $\displaystyle \frac{\partial}{\partial t} k_{s^\ell} = -k_{s^{\ell+4}}+\lambda k_{s^{\ell+2}} - \sum_{j=0}^{\ell} \left[ k\, k_{s^{\ell-j}} (  k_{ss}-\lambda~ k)\, \right]_{s^{j}} \mbox{.}$ 
  \end{itemize}
  \end{lem}
 \begin{lem}\label{T:evlneqns2}
  Under the flow \eqref{eq116}, 
   \begin{enumerate}
   \item[(i)] $\displaystyle{\frac{\mathrm{d} }{\mathrm{d} t}A} =- 2\pi \,\omega \, \lambda $
\item [(ii)] $\displaystyle{\frac{\mathrm{d} }{\mathrm{d} t}L}=-\int _{\alpha }k_{s}^{2}ds-\lambda \int _{\alpha }k^{2}ds;$

\item [(iii)]$\displaystyle{\frac{\mathrm{d} }{\mathrm{d} t}\int_{\alpha } k^{2}ds=-2\int_{\alpha }k_{ss}^2ds+3\int_{\alpha }k^{2}k_{s}^{2}ds+\lambda \int_{\alpha }k^{4}ds-2\lambda \int_{\alpha }k_{s}^{2}ds}$;

\item [(iv)] $\displaystyle{\frac{\mathrm{d} }{\mathrm{d} t}\int _{\alpha }\left ( k-\bar{k} \right )^{2}\,ds=-2\int _{\alpha }k_{ss}^{2}\,ds+3\int _{\alpha }k^{2}k_{s}^{2}\,ds-\bar{k}^{2}\int _{\alpha }k_{s}^{2}\,ds}$

\hspace*{\fill}$\displaystyle{-2\lambda\int _{\alpha }k_{s}^{2}\,ds+\lambda\int _{\alpha }k^{4}\,ds-\lambda \bar{k}^{2}\int _{\alpha }k^{2}\,ds;}$
\item [(v)]$\displaystyle{\frac{\mathrm{d} }{\mathrm{d} t}\frac{1}{2}\int _{\alpha }k_{s^{\ell}}^{2}\,ds=-\int _{\alpha }k_{s^{\ell+2}}^{2}\,ds-\lambda\int _{\alpha }k_{s^{\ell+1}}^{2}\,ds+\int _{\alpha }P_{4}^{2\ell+2}\left ( k \right )\,ds+\lambda \int _{\alpha }P_{4}^{2\ell}\left ( k \right )\,ds}.$
\end{enumerate} 
\end{lem} 
\begin{cor}\label{lambda1}
 Under the flow \eqref{eq116} 
\begin{align*}%\label{osc3}
\frac{\mathrm{d} }{\mathrm{d} t} K_{\mbox{osc}} &= -2L \int_{\alpha} k_{ss}^2 \, ds + 3L \int_{\alpha} (k - \bar{k})^2 k_s^2 \, ds+ 6L \bar{k} \int_{\alpha} (k - \bar{k}) k_s^2 \, ds + 2L \bar{k}^2 \int_{\alpha} k_s^2 \, ds \nonumber   \\
&\quad - \frac{K_{\mbox{osc}}}{L} \int_{\alpha} k_{s}^2 \, ds  - 2 \lambda L \int_{\alpha} k_s^2 \, ds - \lambda \frac{K_{\mbox{osc}}}{L} \int_{\alpha} k^2 \, ds + \lambda L \int_{\alpha} (k - \bar{k})^4 \, ds  \nonumber\\
&\quad + 4 \lambda L \bar{k} \int_{\alpha} (k - \bar{k})^3 \, ds+ 5 \lambda L \bar{k}^2 \int_{\alpha} (k - \bar{k})^2 \, ds
\end{align*} 
\end{cor}
\begin{proof}
The proof is straightforward using Lemma \ref{T:evlneqns2} (ii), (iv), integrating by parts and writing $k=(k-\bar{k})+\bar{k}$. 
\end{proof}
\begin{cor}
Under the flow \eqref{eq116}, while the solution exists we have
\begin{equation}
A\left ( t \right )=A\left ( 0 \right)-2\omega \pi \lambda\, t    
\end{equation}
\end{cor}

\begin{proof} Assuming the solution to \eqref{eq116} exists on the interval $\left[ 0, t\right]$, the result follows by direct integration of Lemma \ref{T:evlneqns2} (i).
\end{proof}

\begin{remark}
Note that under \eqref{eq116} with $\lambda > 0$, while a solution exists the area bounded by the evolving curve and the two sides of the cone decreases at a constant rate, proportional to $\lambda$.  This behaviour is true for all curves, not only the self-similar shrinking circular arcs mentioned earlier. This implies that the maximal time $T$ of existence of a solution to \eqref{eq116} is always finite and may be estimated by
$$T < \frac{A_0}{2\pi\, \omega \lambda} \mbox{,}$$
where $A_0$ is the area enclosed by the initial curve together with the two sides of the cone.  Since $T<\infty$, Theorem \ref{blowup} gives a lower bound on the blow up rate of $\int_{\alpha} k^2 ds$ towards the maximal time.
\end{remark}

Next we establish another estimate on the maximal time of existence $T$.  We remark that such an estimate may be more useful in other situations where an enclosed or generalised area $A$ does not make sense.

\begin{cor}\label{T:Lbound5}
Under the flow \eqref{eq116} with generalised Neumann boundary conditions \eqref{E:NBC1} and \eqref{E:NBC2}, while the solution exists the length of the evolving curve satisfies
 \begin{equation}\label{L}
L^2(t)\leq L^2(0)-2\lambda \left ( 2\omega \pi  \right )^{2}t.  \end{equation}  
\end{cor}
\begin{proof}
By the H\"{o}lder inequality we have
$$2\pi\, \omega = \int_{\alpha } k\, ds \leq L^{\frac{1}{2}} \left( \int_{\alpha } k^2 ds\right)^{\frac{1}{2}} \mbox{.}$$
Discarding the first term in Lemma \ref{T:evlneqns2}, (ii) and using the above inequality on the second term we have
$$\displaystyle{\frac{\mathrm{d} }{\mathrm{d} t}L}\leq -\lambda \frac{\left( 2\pi\, \omega \right)^2}{L}$$
from which \eqref{L} follows by integration.
\end{proof}

\begin{remark}
  From Corollory \ref{T:Lbound5} we obtain an upper bound on the maximal existence time of a solution in terms of $L_0$ and $\lambda$, namely
  $$\displaystyle{T < \frac{L_0^2}{2\lambda(2\omega\pi)^2}}.$$
\end{remark}

Although we will not need it in the small oscillation of curvature setting, we next observe more generally that $K_{\mbox{osc}}$ is $L^{1}$ in time under the flow, which could be useful for future work.

 \begin{proposition} \label{E:L1}
Under the flow \eqref{eq116}, we have   
$$\left\| K_{\mbox{osc}}\right\|_{1} := \int_0^T K_{\mbox{osc}}(t) \,dt \leq \frac{L^4(0)}{4\pi ^{2}}.$$
\end{proposition}

\begin{proof}
Discarding the second term in Lemma \ref{T:evlneqns2} (ii), then the proof follows as in  \cite{HO} and \cite{Wheeler2013}. 
\end{proof}

\begin{proof}[Proof of Proposition~{\upshape\ref{bd}}.] As in \cite{DKS02}, using Lemma \ref{T:evlneqns2} (v), we obtain 
\begin{multline}\label{bb3}
\frac{\mathrm{d} }{\mathrm{d} t}\frac{1}{2}\int _{\alpha }k_{s^{\ell}}^{2}\,ds=-\int _{\alpha }k_{s^{\ell+2}}^{2}\,ds-\lambda\int _{\alpha }k_{s^{\ell+1}}^{2}\,ds+\int _{\alpha }P_{4}^{2\ell+2}\left ( k \right )\,ds +\lambda \int _{\alpha }P_{4}^{2\ell}\left ( k \right )\,ds\\
   \leq -\int _{\alpha }k_{s^{\ell+2}}^{2}\,ds+ \lambda \int _{\alpha }P_{4}^{2\ell}\left ( k \right )\,ds +\int _{\alpha }P_{4}^{2\ell+2}\left ( k \right )\,ds 
\end{multline}
Next, we use the interpolation inequality in Proposition \ref{p}.  For the second term $\lambda \int _{\alpha }P_{4}^{2\ell}(k)\,ds$ in \eqref{bb3}, we obtain 
\begin{equation}\label{v1}\lambda \int _{\alpha }P_{4}^{2\ell}(k)\,ds\leq \varepsilon \int _{\alpha }k_{s^{\ell+2}}^{2}\,ds+c \varepsilon ^{-(\frac{2\ell+1}{3})}\left (\int _{\alpha }k^{2}\,ds  \right )^{\frac{2\ell+7}{3}}+c \left (\int _{\alpha }k^{2}\,ds  \right )^{2\ell+3}.\end{equation} 

For the third term $\int _{\alpha }P_{4}^{2\ell+2}(k)\,ds$ in \eqref{bb3}, we have 
\begin{equation}\label{v2}
 \int _{\alpha }P_{4}^{2\ell+2}(k)ds\leq \varepsilon \int _{\alpha }k_{s^{\ell+2}}^{2}\,ds+c\left ( \varepsilon  \right ) \left (\int _{\alpha }k^{2}\,ds  \right )^{2\ell+5}.
\end{equation}
Substituting these estimates into \eqref{bb3}, we obtain for small $\varepsilon$,

\begin{align}\label{kk22}
\frac{\mathrm{d} }{\mathrm{d} t}\frac{1}{2}\int _{\alpha }k_{s^{\ell}}^{2}\,ds \nonumber
& \leq -\frac{1}{2}\int _{\alpha }k_{s^{\ell+2}}^{2}\,ds+c \, \varepsilon ^{-(\frac{2\ell+1}{3})}\left (\int _{\alpha }k^{2}\,ds  \right )^{\frac{2\ell+7}{3}}\\
&\qquad +c \left (\int _{\alpha }k^{2}\,ds  \right )^{2\ell+3}+ c\left ( \varepsilon  \right )\left (\int _{\alpha }k^{2}\,ds  \right )^{2\ell+5}\nonumber\\
&  \leq-\frac{1}{2}\int _{\alpha }k_{s^{\ell+2}}^{2}\,ds +c( \varepsilon)\left (\int _{\alpha }k^{2}\,ds  \right )^{\frac{2\ell+7}{3}}+c \left (\int _{\alpha }k^{2}\,ds  \right )^{2\ell+3}\nonumber\\&\qquad +c\left ( \varepsilon  \right )\left (\int _{\alpha }k^{2}\,ds  \right )^{2\ell+5}.
\end{align}
Discarding the first term on the right-hand side of \eqref{kk22}, and in view of the boundedness of $\int _{\alpha }k^{2}\,ds$, the result follows by integration over the finite time interval.
\end{proof}

Now we state our main result for this section.
\begin{theorem} \label{T:main18}
Let $\lambda>0$ be constant.  Given $\alpha_0: \left[ -1, 1\right]\to \mathbb{R}^{2}$ compatible with \eqref{E:NBC1} and \eqref{E:NBC2} and with boundary points sufficiently far from the cone tip, the length-penalised curve diffusion flow \eqref{eq116} with generalised Neumann boundary conditions \eqref{E:NBC1} and \eqref{E:NBC2} has a unique solution $\alpha\left( \cdot, t\right)$ on $\left[ 0, T\right)$.  In the case that the supporting cone satisfies 
$$\omega ^{2}\leq \max\left ( \frac{1}{10}, \frac{ \pi ^{2}+\lambda L^2_0}{4{\pi ^{2}}+10\lambda L^2_0} \right )$$ 
and $K_{\mbox{osc}}[\alpha_0] < \overline{K}\left( \omega, \lambda L_0^2 \right)$, where $L_0 = L\left[ \alpha_0 \right]$, then the solution curves (rescaled to preserve length) converge, exponentially in the rescaled time variable, in the $C^\infty$-topology, to a circular arc of length $L_0$.
\end{theorem}

\begin{remark}
\begin{enumerate}
    \item The constant $\overline{K}$ above is computable explicitly, but different depending on the size of $\cfrac{\pi^2+\lambda L^2_0}{4\pi^2+10\lambda L^2_0}$ relative to $\frac{1}{10}$.  In each case it is the positive root of a quadratic polynomial in $x^{\frac{1}{2}}$.  We write this polynomial explicitly in the proof to follow.  The product $\lambda L_0^2$ appears naturally in view of the scaling of terms in the normal speed.
    \item The above `sufficiently far' may be quantified.  Under the small oscillation of curvature assumption, we obtain exponential decay of the speed of evolution under suitable rescaling.  Thus, under the rescaling, the distance each point on the evolving curve can travel is finite, in particular, the boundary points can only travel a finite distance.  Making sure the initial curve has boundary points further from the cone tip than this distance ensures they never reach the cone tip under the rescaled evolution. %in the unrescaled flow.  
   %{ \color{red} Probably just need to reword this to make it clearer.}
\end{enumerate}

\end{remark}

Given the previous results, the next ingredient required to prove Theorem \ref{T:main18} is control on $K_{\mbox{osc}}$.
\begin{cor}  Under the flow \eqref{eq116} 
 \begin{align}\label{ekosc2}
\frac{\mathrm{d} }{\mathrm{d} t} K_{\mbox{osc}} &\leq -L \Bigg\{ \left [  2 - 8 \omega^2 \right ] -8\omega \left[ 3+ \frac{2\lambda L^2}{\pi^2}  \right] \sqrt{K_{\mbox{osc}}} -\frac{2}{\pi }\left [ 3+\frac{\lambda L^{2}}{\pi ^{2}} \right ] K_{\mbox{osc}} \Bigg\}\int_{\alpha} k_{ss}^2 \, ds\nonumber\\
 & \quad +2\lambda L\left (10\omega ^{2}-1  \right )\int _{\alpha }k_{s}^{2}\,ds
\end{align} 
\end{cor}
\begin{proof}
We begin by discarding the fifth and seventh nonpositive terms on the right-hand side of Corollary \ref{lambda1}.  Then, similarly as in \cite{WW}, using Proposition \ref{psw} \eqref{eqpsw},  \eqref{eqppsw}, and the H\"{o}lder inequality, we obtain
$$3L\int _{\alpha }\left ( k-\bar{k} \right )^{2}k_{s}^{2}ds\leq \frac{6L}{\pi }K_{\mbox{osc}}\left\|k_{ss} \right\|_{2}^{2},$$
 $$6\bar{k}L\int _{\alpha }\left ( k-\bar{k} \right )k_{s}^{2}ds\leq 24\omega L \sqrt{K_{\mbox{osc}}} \left\|k_{ss} \right\|_{2 }^{2},$$
 and\begin{align*}
 2L\bar{k}^{2}\left\|k_{s} \right\|_{2}^{2}&\leq 8\omega ^{2}L\left\|k_{ss} \right\|_{2}^{2}
\end{align*}  
Similarly,
\begin{equation}\label{osc5e}
5\lambda L\bar{k}^{2}\int _{\alpha }\left ( k-\bar{k} \right )^{2}\,ds\leq 20\lambda L\omega ^{2}\int _{\alpha }k_{s}^{2}\,ds    
\end{equation}
$$4 \lambda L \bar{k} \int (k - \bar{k})^3 \, ds\leq 16 \lambda \omega   \frac{L^3}{\pi^2}   \sqrt{K_{\mbox{osc}}}\int_{\alpha } k_{ss}^2 \, ds$$
\begin{multline*}
\lambda L \int (k - \bar{k})^4 \, ds 
\leq \lambda L \| k - \bar{k} \|^2_{\infty } \int (k - \bar{k})^2 \, ds \\
\leq   \frac{2\lambda L}{\pi} \,K_{\mbox{osc}} \int_{\alpha } k_{s}^{2}  \, ds 
\leq \frac{2 \lambda L^3}{\pi^3}\,K_{\mbox{osc}} \int_{\alpha } k_{ss}^2 \, ds
\end{multline*}
Combining the sixth term on the right-hand side of Corollary \ref{lambda1} with \eqref{osc5e} 
gives
$$2\lambda L\left (10\omega ^{2}-1  \right )\int _{\alpha }k_{s}^{2}\,ds $$
Using these estimates in Corollary \ref{lambda1} completes the proof.    
\end{proof}

\begin{cor}
    Under the flow \eqref{eq116}, in the case $\omega^2 \leq \frac{1}{10}$ we have
\begin{equation*}
    \frac{\mathrm{d} }{\mathrm{d} t} K_{\mbox{osc}} \leq -L \left\{ \left [  2 - 8 \omega^2 \right ] -8\omega \left[ 3+ \frac{2\lambda L_0^2}{\pi^2}  \right] \sqrt{K_{\mbox{osc}}} 
   -\frac{2}{\pi }\left [ 3+\frac{\lambda L_0^{2}}{\pi ^{2}} \right ] K_{\mbox{osc}} \right\}\int_{\alpha} k_{ss}^2 \, ds
\end{equation*}
    while if $\omega ^{2}>  \frac{1}{10}$ we have
     \begin{multline*}%\label{ekosc2}
\frac{\mathrm{d} }{\mathrm{d} t} K_{\mbox{osc}} \leq -L \Bigg\{ \left [  2 - 8 \omega^2 - 2\left( 10 \omega^2 - 1\right) \frac{\lambda L_0^2}{\pi^2}\right ] \\-8\omega \left[ 3+ \frac{2\lambda L_0^2}{\pi^2}  \right] \sqrt{K_{\mbox{osc}}} 
 -\frac{2}{\pi }\left [ 3+\frac{\lambda L_0^{2}}{\pi ^{2}} \right ] K_{\mbox{osc}} \Bigg\}\int_{\alpha} k_{ss}^2 \, ds
\end{multline*} 
It follows in each case that if $K_{\mbox{osc}}\left[ \alpha_0\right]$ is sufficiently small, say $K_{\mbox{osc}}\left[ \alpha_0\right] \leq \overline{K}\left( \omega, \lambda L_0^2\right)$, then it remains so.
\end{cor}
\begin{proof}
 In the case $\omega^2 \leq \frac{1}{10}$ we neglect the negative $\int_\alpha k_s^2\, ds$ term on the right-hand side of \eqref{ekosc2}.  In the other case, this same term is positive, so we instead apply Proposition \ref{psw} and include the resulting term in the coefficient of $\int_\alpha k_{ss}^2\,ds$.  In both cases the inequalities are then established by estimating $L^2 \leq L_0^2$ on the inside factors.

    The proof is completed in each case by noting that if the coefficient of $\int_\alpha k_{ss}^2\,ds$ is nonpositive, then $K_{\mbox{osc}}$ does not increase.  Explicitly, in the case $\omega^2 < \frac{1}{10}$, this will be true provided $K_{\mbox{osc}}\left[ \alpha_0 \right]$ is not larger than the solution $x$ of
\begin{equation} \label{E:S1}
  \left [  2 - 8 \omega^2 \right ] -8\omega \left[ 3+ \frac{2\lambda L_0^2}{\pi^2}  \right] x^{\frac{1}{2}}  -\frac{2}{\pi }\left [ 3+\frac{\lambda L_0^{2}}{\pi ^{2}} \right ] x =0\mbox{,}
  \end{equation}
    while in the other case $K_{\mbox{osc}}\left[ \alpha_0 \right]$ should be not larger than the solution of
\begin{equation} \label{E:S2}
  \left [  2 - 8 \omega^2 - 2\left( 10 \omega^2 - 1\right) \frac{\lambda L_0^2}{\pi^2} \right ] -8\omega \left[ 3+ \frac{2\lambda L_0^2}{\pi^2}  \right] x^{\frac{1}{2}}  -\frac{2}{\pi }\left [ 3+\frac{\lambda L_0^{2}}{\pi ^{2}} \right ] x =0\mbox{.}
  \end{equation}
    \end{proof}
\begin{remark}
Since
 $$ \int _{\alpha }k^{2}\,ds=\frac{ K_{\mbox{osc}}}{L}+\frac{4\pi ^{2}\omega ^{2}}{L}$$
we observe that the boundedness of $K_{\mbox{osc}}$ under the flow implies that $\int_{\alpha } k^2 ds$ remains bounded unless $L\searrow 0$.  Hence, in view of the preserved smallness condition on $K_{\mbox{osc}}$ we conclude existence of solutions until the time that $L\searrow 0$.  In other words, the flow can be extended using short-time existence as long as $L>0$ and it makes sense to consider a rescaling of the evolving curve to fix length and examine the corresponding limiting shape.
\end{remark}
To prepare for rescaling, we first use Proposition \ref{psw} twice in reverse to obtain the following:
\begin{cor} \label{T:Koscexp}
Suppose $\alpha_0$ satisfies $K_{\mbox{osc}}\left[ \alpha_0 \right] < \overline{K}$.  Under the flow \eqref{eq116}, in the case $\omega^2 \leq \frac{1}{10}$ we have
\begin{equation*}
    \frac{\mathrm{d} }{\mathrm{d} t} K_{\mbox{osc}} \leq -\frac{\pi^4}{L^4} \left\{ \left [  2 - 8 \omega^2 \right ] -8\omega \left[ 3+ \frac{2\lambda L_0^2}{\pi^2}  \right] \sqrt{K_{\mbox{osc}}} -\frac{2}{\pi }\left [ 3+\frac{\lambda L_0^{2}}{\pi ^{2}} \right ] K_{\mbox{osc}} \right\}K_{\mbox{osc}}
\end{equation*}
    while in the case $\omega ^{2}>  \frac{1}{10}$ we have
     \begin{multline*}
\frac{\mathrm{d} }{\mathrm{d} t} K_{\mbox{osc}} \leq -\frac{\pi^4}{L^4} \Bigg\{ \left [  2 - 8 \omega^2 - 2\left( 10 \omega^2 - 1\right) \frac{\lambda L_0^2}{\pi^2}\right ] \\-8\omega \left[ 3+ \frac{2\lambda L_0^2}{\pi^2}  \right] \sqrt{K_{\mbox{osc}}} 
 -\frac{2}{\pi }\left [ 3+\frac{\lambda L_0^{2}}{\pi ^{2}} \right ] K_{\mbox{osc}} \Bigg\}K_{\mbox{osc}}
\end{multline*} 
\end{cor}

Let us now define the rescaled time variable $\tilde t$ via $\frac{\mathrm{d} \tilde t}{\mathrm{d} t}= L^{-4}$.  By slight abuse of notation, let us now consider scale invariant quantities as functions of $\tilde t$ where no confusion arises.
\begin{cor} \label{kosc5}
     Under the conditions of Theorem \ref{T:main18}, the oscillation of curvature of the evolving curve satisfies
    $$K_{\mbox{osc}}(\tilde t) \leq C e^{-\delta \tilde t} \mbox{,}$$
    where $\delta = \delta\left( \omega, \lambda L_0^2\right)$ and $C=K_{\mbox{osc}}(0).$
\end{cor}
\begin{proof}
    Suppose that instead of the smallness requirements \eqref{E:S1} and \eqref{E:S2} we have that each of these quantities is equal to $\delta>0$, that is, slightly stronger smallness requirements that can be accommodated in view of the strict inequality in the statement of Theorem \ref{T:main18}.  Then, by the previous argument, such smallness is preserved and Corollary \ref{T:Koscexp} implies in each case
$$\frac{\mathrm{d} }{\mathrm{d} t} K_{\mbox{osc}} \leq -\frac{\pi^4}{L^4} \delta K_{\mbox{osc}} \mbox{.}$$
This means precisely that
$$\frac{\mathrm{d} }{\mathrm{d} \tilde t} K_{\mbox{osc}} \leq -\pi^4 \delta K_{\mbox{osc}}$$
from which exponential decay of $K_{\mbox{osc}}$ in the $\tilde t$ variable follows, redefining the $\delta$ to incorporate $\pi^4$.
\end{proof}

\begin{remark}
 We have chosen not to work with the evolution equations for the rescaled flow directly in this article, since it is more efficient to simply convert rescaled quantities back to the original variables and use the existing evolution equations from Lemma \ref{T:evlneqns1} and so on.  However, for completeness, we note that the rescaled curves are given by $\tilde{\alpha } \left ( \cdot, t \right )=\frac{1}{L(t)}\alpha \left ( \cdot, t \right )$ and the rescaled flow equation is \begin{equation}\label{eq101}
\frac{\partial \tilde{\alpha }}{\partial \tilde{t}}=-\tilde{F}\tilde{\nu }+\tilde{\alpha }\tilde{h}  \end{equation}
where $\tilde{h}=\int _{\tilde{\alpha} }\tilde{k}\,\tilde{F}\,d\tilde{s}=L^4\int _{\alpha }k\,F\,ds$,  $\tilde{t}=\int_{0}^{t}\frac{1}{L^4(\tau)}\,d\tau, \tilde{F}=L^3\,F$, $d\tilde{s}=\frac{1}{L}\,ds$, $\tilde{\lambda }=\lambda \,L^{2} $ and  $\tilde{\nu }=\nu$, as in \cite{Huisken84}, for example.
We further note that the rescaled curves continue to satisfy the boundary conditions \eqref{E:NBC1} and \eqref{E:NBC2}.
\end{remark}

Next, we show that the scale invariant $L^2$ curvature derivative norms are all bounded provided $L \int_{\alpha} k^2 ds$ is bounded.  While we could compute all the evolution equations in the rescaled coordinates, it is easier to convert back into the original unrescaled quantities and use the existing evolution equations.  

\begin{cor} \label{T:rescaled}
If, under the flow \eqref{eq101}, $\int_{\tilde{\alpha}} \tilde{k}^2 d\tilde{s}$ is bounded on a finite time interval, then all the rescaled curvature derivatives in $L^2$ are bounded on any finite time interval; more precisely
$$\left ( L^{2\ell+1}\int _{\alpha }k_{s^\ell}^{2}\,ds \right )\bigg{|}_{\tilde t}\leq \left ( L^{2\ell+1}\int _{\alpha }k_{s^\ell}^{2}\,ds \right )\bigg{|}_{\tilde t=0}+\tilde{C}_{\ell}^{2\ell+5}\,\tilde t$$
\end{cor}
\begin{proof}
We calculate
\begin{align*}
&\frac{\mathrm{d} }{\mathrm{d} \tilde{t}}\int _{\alpha }\tilde{k}_{\tilde{s}^\ell}^{2}\,d\tilde{s}
%=\frac{\mathrm{d} }{\mathrm{d} \tilde{t}}\left ( L^{2\ell+1}\int _{\alpha }k_{s^\ell}^{2}\,ds \right )
=\frac{\mathrm{d} }{\mathrm{d} t}\left ( L^{2\ell+1}\int _{\alpha }k_{s^\ell}^{2}\,ds \right )\frac{\mathrm{d} t}{\mathrm{d} \tilde{t}} 
=L^4\bigg [\left ( 2\ell+1 \right )L^{2\ell}\,\frac{\mathrm{d} L}{\mathrm{d} t} \int _{\alpha }k_{s^\ell}^{2}\,ds +L^{2\ell+1}\frac{\mathrm{d} }{\mathrm{d} t}\int _{\alpha }k_{s^\ell}^{2}\,ds\bigg ]
 \\
 &\leq L^{2\ell+5}\frac{\mathrm{d} }{\mathrm{d} t}\int _{\alpha }k_{s^\ell}^{2}\,ds
\leq L^{2\ell+5}\bigg [ c(\varepsilon)\left (\int _{\alpha }k^{2}\,ds  \right )^{\frac{2\ell+7}{3}} +c \left (\int _{\alpha }k^{2}\,ds  \right )^{2\ell+3} +c\left ( \varepsilon  \right )\left (\int _{\alpha }k^{2}\,ds  \right )^{2\ell+5} \bigg ]\\
&\leq c(\varepsilon)L^{\frac{4\ell+8}{3}}\left (\int _{\tilde{\alpha} }\tilde{k}^{2}\,d\tilde{s}  \right )^{\frac{2\ell+7}{3}} +cL^2 \left (\int _{\tilde{\alpha} }\tilde{k}^{2}\,d\tilde{s}  \right )^{2\ell+3} +c\left ( \varepsilon  \right )\left (\int _{\tilde{\alpha} }\tilde{k}^{2}\,d\tilde{s}  \right )^{2\ell+5}.
\end{align*}
In view of the boundedness of $\int_{\tilde{\alpha}} \tilde k^2 d\tilde s$, the result follows by integration over the finite time interval.
\end{proof}
%\begin{remark}
%Corollary \ref{T:rescaled} shows more generally that while $\int_{\tilde{\alpha}} \tilde k^2 d\tilde s$ is bounded the rescaled curvature derivatives can each grow at most linearly with respect to $\tilde t$.  This is sufficient to obtain exponential decay in the argument to follow. {\color{red} Remove it?}
%\end{remark}

Now we obtain exponential decay in the new time variable of all the scale invariant $L^2$ curvature derivatives.  The first of these is a special case that we state separately.

\begin{cor}
 Under the flow \eqref{eq101}, for all $\tilde t >0$,
 \begin{equation}\label{exc175}
L^3\left\| k_{s} \right\|_2^2 \leq \tilde{C}\, e^{-\frac{\delta\, \tilde t}{4}} 
\end{equation}    
\end{cor}
\begin{proof}Using integration by
parts, Lemma \ref{T:BCs}, and the H\"{o}lder inequality yields
 \begin{equation*}
L^3 \int_{\alpha } k_s^2 \, ds = -L^3 \int_{\alpha } \left( k - \bar{k} \right) k_{ss} \, ds 
\leq \left( L^5 \int_{\alpha } k_{ss}^2 \, ds \right)^{\frac{1}{2}} \left( L \int_{\alpha } \left( k - \bar{k} \right)^2  ds \right)^{\frac{1}{2}}.
\end{equation*} 
Using Corollaries \ref{kosc5} and \ref{T:rescaled}, we obtain
\begin{equation*}
L^3 \int_{\alpha } k_s^2 \, ds \leq  \left ( \tilde{C_1}+\tilde{C_2^9}\,\tilde{t} \right )^{\frac{1}{2}}C^{\frac{1}{2}}\,e^{-\frac{\delta }{2}\tilde{t}} 
%= \left ( \left ( \tilde{C_1}+\tilde{C_2^9}\,\tilde{t} \right )^{\frac{1}{2}}C^{\frac{1}{2}}\,e^{-\frac{\delta }{4}\tilde{t}} \right )e^{-\frac{\delta }{4}\tilde{t}},
\leq \tilde{C}\, e^{-\frac{\delta }{4}\tilde{t}},
\end{equation*}
where $\tilde{C}_1:=\left ( L^{5}\int _{\alpha }k_{ss}^{2}\,ds \right )\bigg{|}_{\tilde t=0}$.
%Since $\left ( \tilde{C_1}+\tilde{C_2}\tilde{t} \right )^{\frac{1}{2}}C^{\frac{1}{2}}\,e^{-\frac{\delta }{4}\tilde{t}}\leq \tilde{C} $, for a constant $\tilde{C}>0$. Thus,  $$L^3 \int_{\alpha } k_s^2 \, ds\leq \tilde{C}\, e^{-\frac{\delta }{4}\tilde{t}},$$
This completes the proof.
\end{proof}

The higher curvature derivatives can now be handled by induction.

\begin{cor}\label{cor3.14}
 Under the flow \eqref{eq101}, for all $\ell \in \mathbb{N}$, there exist corresponding $\delta_\ell >0$ and $\tilde C_\ell >0$ such that
 \begin{equation}
L^{2\ell+1}\left\| k_{s^\ell} \right\|_2^2 \leq \tilde{C}_{\ell}\, e^{-\delta_{\ell} \tilde t} \mbox{.}    \end{equation}    
\end{cor}
\begin{proof}
 The result follows from a standard induction argument, with the base case $\ell = 1$ given in \eqref{exc175}. So assume that
 $$L^{2\ell+1}\int _{\alpha }k_{s^\ell}^{2}\,ds\leq \tilde C_\ell e^{-\delta_\ell \tilde{t}}$$ for some $\tilde{C_\ell}>0$.  Integrating by parts and using the property at the boundaries, Lemma \ref{T:BCs},   we estimate using the H\"{o}lder inequality
 \begin{equation*}
L^{2\ell+3}\int _{\alpha }k_{s^{\ell+1}}^2\,ds
= -L^{2\ell+3}\int _{\alpha }k_{s^{\ell+2}}\,k_{s^{\ell}}\,ds 
\leq \left( L^{2\ell+1}\int _{\alpha }k_{s^{\ell}}^2\,ds \right)^{\frac{1}{2}}
\left( L^{2} \int _{\alpha }k_{s^{\ell+2}}^2\,ds \right)^{\frac{1}{2}}
\end{equation*} 
The inductive assumption and Corollary \ref{T:rescaled}, 
imply
$$L^{2\ell+3}\int _{\alpha }k_{s^{\ell+1}}^2\,ds\leq C_{\ell+2}\, \tilde{C_\ell}^\frac{1}{2}e^{-\frac{\delta_\ell}{2} \tilde{t}}.$$
The proof is concluded by setting $\tilde C_{\ell+1} := C_{\ell+2} \tilde C_\ell^{\frac{1}{2}}$ and $\delta_{\ell+1} := \frac{\delta_\ell}{2}$.
\end{proof}

\begin{proof}[Completion of the proof of Theorem \ref{T:main18}]
The exponential decay in $L^2$ with respect to the time variable $\tilde t$ implies pointwise exponential decay of the rescaled quantities via Proposition \ref{psw}.  In particular, pointwise exponential decay of $K_{\mbox{osc}}$ implies convergence of the rescaled curves in the rescaled time variable to a circular arc of length $L_0$ with centre at the tip of the cone, compatible with the boundary conditions \eqref{E:NBC1} and \eqref{E:NBC2}.  To establish exponential decay of all curvature derivatives with the same exponent one may use a standard linearisation argument, noting that, by continuity, there will exist a time beyond which the solution remains strictly convex and thus parametrisation by the support function may be used (notwithstanding a more general procedure using the position vector is also possible).  Moreover, again in view of the exponential convergence, one may obtain a time $\tilde t_0$ at which the rescaled curve is $C^{4,\gamma}$-close to the limiting curve $\tilde \alpha_\infty$, as required to apply the linearisation argument.  Finally, convergence of the rescaled solution in the $C^\infty$-topology with the rescaled time variable follows by standard arguments, using control on the rescaled speed to control parameter derivatives of the embedding map.  We refer the reader to \cite{GM24} for an outline of this procedure, which is similar, although there not in rescaled variables. This completes the proof of the Theorem.
\end{proof}

\section{Length-constrained curve diffusion flow \texorpdfstring{$F= -k_{ss} +\lambda_1(t)$}{lambda(t) } } \label{S:lf}
In this section, we choose $\lambda$ as a function of time that ensures the length of the evolving curve is fixed. The corresponding flow equation is 
\begin{equation}\label{eq113}
\frac{\partial \alpha }{\partial t}=\bigg(  k_{ss}-\lambda_1(t) \bigg)\,\nu 
\end{equation}
with \begin{equation}\label{h1}
\lambda_1\left ( t \right ):=\displaystyle{\frac{-\int _{\alpha }k_{s}^{2}\,ds}{\int _{\alpha }k\,ds}=\frac{-\int _{\alpha }k_{s}^{2}\,ds}{ 2\omega \pi }}
\end{equation}
ensures $L[\alpha_t] = L[\alpha_0] =: L_0$ for all $t$, as can be checked by direct calculation (Lemma \ref{T:evlneqns48}).  This is the same flow speed as used in the case of closed planar curves in \cite{MWY}. In this case, we have all odd curvature derivatives equal to zero on the boundary (i.e., $ k_{s^{2\ell+1}}=0$ on the boundary). As such, by the same interpolation argument using integration by parts but now with boundary terms that are equal to zero, we have the following estimate. 

\begin{lem}\label{c2}For each $n \in \mathbb{N} $
$$\left | \lambda_1 (t)\right |\leq \frac{1}{2\pi\omega } \left( \int_{\alpha } k^2 \, ds \right)^{1 - \frac{1}{n}} \left( \int_{\alpha } k_{s^n}^2 \, ds \right)^{\frac{1}{n}}.$$ 
\end{lem}

Again, short-time existence of a solution to the flow equation \eqref{eq113} is well known, with a classical fixed point argument used to handle the global term $\lambda_1(t)$.
\begin{lem} \label{T:evlneqns48}
Under the flow \eqref{eq113}, while a solution exists, the length is constant  
$L\left [ \alpha_{t}  \right ]=L_{0}.$
\end{lem}
\begin{proof}
By differentiating the length in time, we get
$$\frac{\mathrm{d} }{\mathrm{d} t}L\left [ \alpha_{t}  \right ]=- \int_{\alpha }kF~ds=-\int_{\alpha }k_{s}^{2}\,ds-\lambda_1 (t)\int_{\alpha }k~ds = 0,$$ 
where in the last step we have used \eqref{h1}.  The result follows.      \end{proof}
\begin{lem} \label{h32}
Under the flow \eqref{eq113}, we have 
\begin{enumerate}
 \renewcommand{\labelenumi}{(\roman{enumi})}
 \item $\displaystyle{\frac{\mathrm{d} }{\mathrm{d} t}A\left [ \alpha_t \right ]=-\lambda_1 (t) L_0};$
 \item $\displaystyle{\frac{\mathrm{d} }{\mathrm{d} t} \int _{\alpha } k^2 \, ds = -2 \int _{\alpha } k_{ss}^2 \, ds 
+ 3 \int _{\alpha } k^2 k_s^2 \, ds + \lambda_1(t) \int _{\alpha } k^3 \, ds};$
\item $\displaystyle{\frac{\mathrm{d} }{\mathrm{d} t} \int _{\alpha } k_s^2 \, ds = -2 \int _{\alpha } k_{s^3}^2 \, ds 
+ 2 \int _{\alpha } k^2 k_{ss}^2 \, ds + \frac{1}{3} \int _{\alpha } k_s^4 \, ds + 5\lambda_1(t) \int _{\alpha } k k_s^2 \, ds};$
\item $\displaystyle{\frac{\mathrm{d} }{\mathrm{d} t} K_{\mbox{osc}} = -2L_0\int _{\alpha }k_{ss}^2 \, ds + 3L_0 \int _{\alpha } (k - \overline{k})^2 k_s^2 \, ds 
+ 6L_0 \overline{k} \int _{\alpha } (k - \overline{k}) k_s^2 \, ds}$

\hspace*{\fill}$\displaystyle{+2\overline{k}^2 L_0 \int k_s^2 \, ds + L_0 \lambda_1(t) \left[ \int _{\alpha } (k - \overline{k})^3 \, ds 
+ 3\overline{k} \int _{\alpha } (k - \overline{k})^2 \, ds \right]};$
\item $\displaystyle{\frac{\mathrm{d} }{\mathrm{d} t}\frac{1}{2}\int _{\alpha }k_{s^{\ell}}^{2}\,ds=-\int _{\alpha }k_{s^{\ell+2}}^{2}\,ds+\int _{\alpha }P_{4}^{2\ell+2}\left ( k \right )\,ds+\lambda_1 \left ( t \right )\int _{\alpha }P_{3}^{2\ell}\left ( k \right )\,ds.}$\\
\end{enumerate}
\end{lem}
\begin{proof}[Proof of Proposition~{\upshape\ref{bd}}.]
The proof follows, as in \cite{DKS02} and \cite{MWY}, using Lemmas  \ref{c2} and \ref{h32} (v).
\end{proof}

As in \cite{MWY}, under the flow \eqref{eq113}, we show that $K_{\mbox{osc}}$ is a $L^1$ function in time by using Proposition  \ref{psw}  \eqref{eqpsw} and Lemma \ref{h32} (i); however, here we do not use the isoperimetric ratio. 
  \begin{proposition} \label{T:KoscL1}
Under the flow \eqref{eq113}, we have   
$$\left\| K_{\mbox{osc}}\right\|_{1}\leq \frac{2\omega L_{0}^{2}}{\pi }\left ( \frac{L_{0}^{2}}{2 }-A\left ( 0 \right ) \right ),$$
where $A(0)$ corresponds to the signed enclosed area of $\alpha_0$.
\end{proposition}
\begin{proof}
\begin{equation*}
A = \frac{1}{2} \int_{\alpha} < \alpha, \nu > ds   
 = \frac{1}{2} \int_{\alpha} < \alpha - \overline{\alpha}, \nu > ds\leq \frac{1}{2} \left| \alpha - \overline{\alpha} \right|_\infty \int_{\alpha} 1\, ds
\end{equation*} 
where $\overline{\alpha}$ is the point whose coordinates are the average values of $x$ and of $y$ over the curve $\alpha$.
And $\int_{\alpha} 1\, ds = L$ while $\left| \alpha - \overline{\alpha} \right|_\infty \leq L$, because no point on the curve can be further than $L$ from the average value of the coordinates over the curve. Now, $ L$ is constant so we have
$$A(t) \leq \frac{1}{2} L_0^2$$

Estimate $$K_{\mbox{osc}}=L_0\int _{\alpha }\left ( k-\bar{k} \right )^{2}\,ds\leq \frac{ L_{0}^{3}}{\pi ^{2}}\left\| k_{s}\right\|_{2}^{2}=\frac{2\omega L_{0}^{2}}{\pi }\frac{\mathrm{d} A}{\mathrm{d} t}.$$

Therefore,
$$\left\| K_{\mbox{osc}}\right\|_{1}\leq \frac{2\omega L_{0}^{2}}{\pi }\bigg ( A\left ( t \right )-A\left ( 0 \right ) \bigg )\leq \frac{2\omega L_{0}^{2}}{\pi }\left ( \frac{L_{0}^{2}}{2 }-A\left ( 0 \right ) \right ).$$
\end{proof}
\begin{remark}
In view of our smallness conditions, we do not require Proposition \ref{T:KoscL1} subsequently.  However, we have included it for potential future use, in other scenarios with less restrictive smallness conditions.
\end{remark}

We have the following main result for this section.

 \begin{theorem} \label{T:main2}
 Suppose the supporting cone has $\omega < \frac{1}{2}$.  Let $\alpha_0: \left[ -1, 1\right]\to \mathbb{R}^{2}$, compatible with \eqref{E:NBC1} and \eqref{E:NBC2}, be a given initial curve of length $L_0$ whose boundary points are sufficiently far from the cone tip.  Assume further that $\alpha_0$ has $K_{\mbox{osc}}$ less than  $x>0$, where $x^{\frac{1}{2}}$ is the smallest solution of
 \begin{equation}\label{s1K_{OSC}}
 \left( 2 - 8 \omega^2 \right) - \frac{\sqrt{2}}{2\pi^{5/2} \omega} x^{\frac{3}{2}} - \frac{6}{\pi} x - 24\omega x^{\frac{1}{2}} = 0    \mbox{.}
 \end{equation}
  Then the length-constrained curve diffusion flow \eqref{eq113} with generalised Neumann boundary conditions \eqref{E:NBC1} and \eqref{E:NBC2} and initial condition $\alpha_0$ has a unique solution $\alpha\left( \cdot, t\right)$ on $\left[ 0, \infty\right)$. The solution curves converge smoothly and exponentially in the $C^\infty$-topology to the circular arc $\alpha_\infty$ of length $L_0$ centred at the cone tip.\\
\end{theorem}
\begin{remark}
\begin{enumerate}
 \item The smallness requirement is the solution of a cubic equation in $y=x^{\frac{1}{2}}$. In view of the signs of the coefficients in \eqref{s1K_{OSC}}, there is one real solution.  By setting $y=x^{\frac{1}{2}}$ in \eqref{s1K_{OSC}} we obtain the following cubic equation 
$$\left(  2- 8 \omega^2\right)- \frac{\sqrt{2}}{2\pi^{5/2} \omega} y^3 - \frac{6}{\pi} y^2 - 24\omega y  = 0$$
with $
a =- \frac{\sqrt{2}}{2\pi^{5/2} \omega}, b=- \frac{6}{\pi},
c = -24\omega~~ \textit{and}~~d =  2- 8\omega^2.$ Now using Cardano’s formula gives the positive real root of the cubic as:
\begin{equation}\label{y}
y = -\frac{1}{3a} \left(b + C + \frac{\Delta_0}{C} \right),    
\end{equation} where
\begin{equation}\label{c}
C = \sqrt[3]{\frac{\Delta_1 + \sqrt{\Delta_1^2 - 4\Delta_0^3}}{2}},    
\end{equation}
and
\begin{align*}
\Delta_0 &= b^2 - 3ac =36 \left ( \frac{1}{\pi^2} - \frac{\sqrt{2}}{\pi^{5/2}} \right ) \\
\Delta_1 &= 2b^3 - 9abc + 27a^2d =- \frac{432}{\pi^3} - \frac{648\sqrt{2}}{\pi^{7/2}} + \frac{27(1-4\omega^2)}{\pi^5 \omega^2}. 
\end{align*}
  \item The 'sufficiently far from the cone tip' requirement can again be quantified via the same argument as in the remark of the previous section.
\end{enumerate}
\end{remark}

Again, in view of the results in Section 2, the key point is to control $K_{\mbox{osc}}$ under the flow.
\begin{lem}\label{lk1}
   Under the flow \eqref{eq113},
 \begin{equation}
\frac{\mathrm{d} }{\mathrm{d} t} K_{\mbox{osc}} + L_{0} \left( 2 - \frac{\sqrt{2}}{2\pi^{5/2} \omega} K_{\mbox{osc}}^{\frac{3}{2}} - \frac{6}{\pi} K_{\mbox{osc}} - 24\omega K_{\mbox{osc}}^{\frac{1}{2}}-8 \omega^2\right) 
\int_{\alpha} k_{ss}^2 \, ds \leq 0. 
\end{equation}
Consequently, if $\omega < \frac{1}{2}$ and $K_{\mbox{osc}}\left[ \alpha_0\right]$ is no bigger than the solution $x$ of \eqref{s1K_{OSC}}, then $K_{\mbox{osc}}$ does not increase under the flow.
\end{lem}
\begin{proof}
We begin with the evolution equation of  Lemma \ref{h32} (iv), we estimate using Proposition \ref{psw} \eqref{eqpsw} and \eqref{eqppsw} as in \cite{MWY} and \cite{Wheeler2013} 
$$3L_0\int _{\alpha }\left ( k-\bar{k} \right )^{2}k_{s}^{2}\,ds\leq \frac{6L_0}{\pi }K_{\mbox{osc}}\left\|k_{ss} \right\|_{2}^{2},$$
 $$6\bar{k}L_0\int _{\alpha }\left ( k-\bar{k} \right )k_{s}^{2}\,ds\leq 24\omega L_0 \sqrt{K_{\mbox{osc}}} \left\|k_{ss} \right\|_{2 }^{2}$$
 and\begin{align*}
 2L_0\bar{k}^{2}\left\|k_{s} \right\|_{2}^{2}&\leq 8\omega ^{2}L_0\left\|k_{ss} \right\|_{2}^{2}.
\end{align*}
While we discard the last non-positive term that is not useful in Lemma \ref{h32} (iv), we estimate the term $L_0 \lambda_1(t)  \int _{\alpha } (k - \overline{k})^3 \, ds $ as follows:
$$ \int_{\alpha} k_s^2\, ds = - \int_{\alpha} k\, k_{ss} \, ds 
= - \int_{\alpha} (k - \bar{k})\, k_{ss} \, ds 
\leq \left( \int_{\alpha} (k - \bar{k})^2 \, ds \right)^{\frac{1}{2}} \left( \int_{\alpha} k_{ss}^2 \, ds \right)^{\frac{1}{2}} $$
and 
$$\int_{\alpha} (k - \bar{k})^3 \, ds 
\leq \|k - \bar{k}\|_{\infty} \int_{\alpha} (k - \bar{k})^2 \, ds;$$
therefore,
$$L_0\, \lambda_1(t) \int_{\alpha} (k - \bar{k})^3 \, ds\leq L_0\, \left | \lambda_1(t)\right | \int_{\alpha} (k - \bar{k})^3 \, ds$$
By using Lemma \ref{c2}, when $n=1$, we have 
\begin{equation*}
L_0\, \lambda_1(t) \int_{\alpha} (k - \bar{k})^3 \, ds\leq L_0\, \left | \lambda_1(t)\right | \int_{\alpha} (k - \bar{k})^3 \, ds  
  \leq \frac{L_0}{2\omega \pi}\int_{\alpha } k_{s}^{2}\,ds \int_{\alpha} (k - \bar{k})^3 \, ds  
\end{equation*}
thus,
\begin{equation*}
L_0\, \lambda_1(t) \int_{\alpha} (k - \bar{k})^3 \, ds 
\leq \frac{L_0}{2\pi\omega} 
 \left( \int_{\alpha} (k - \bar{k})^2 \, ds \right)^{\frac{3}{2}} \|k - \bar{k}\|_{\infty}  
\|k_{ss}\|_2 
\leq \frac{L_0}{\pi^2 \omega \sqrt{2\pi}} K_{\mbox{osc}}^{\frac{3}{2}} \|k_{ss}\|_2^2.
\end{equation*}
By inserting these estimates into  Lemma  \ref{h32} (iv), we complete the proof.    
\end{proof}

As in the previous section, the bound on $K_{\mbox{osc}}$ implies $\int_\alpha k^2 ds$ is bounded.  Since $L$ is constant here, we deduce $T=\infty$.

A refinement of the proof of Lemma \ref{lk1} now yields a crucial exponential decaying quantity under the flow, namely, $\int_{\alpha } \left( k - \overline{k}\right)^2 ds$.
\begin{cor}\label{K} 
Suppose $\omega < \frac{1}{2}$ and $\alpha_0$ is such that $K_{\mbox{osc}}\left[ \alpha_0\right]$ is strictly less than the solution $x$ of \eqref{s1K_{OSC}}.  Then, under the flow \eqref{eq113},
$$\int _{\alpha }\left( k-\overline{k}\right)^2 ds =\frac{K_{\mbox{osc}}\left( t\right)}{L_0} \leq \delta_1 e^{- \frac{\delta \pi^4}{L_0^4} t} \mbox{},$$
where $\delta_1=\displaystyle{\frac{K_{\mbox{osc}}\left( 0\right)}{L_0}}\,\mbox{.}$
\end{cor}
\begin{proof}
In view of the strict inequality we have $\delta >0$ such that
\begin{equation}
2 -\frac{\sqrt{2}}{2\pi^{5/2} \omega} K_{\mbox{osc}}^{\frac{3}{2}} - \frac{6}{\pi} K_{\mbox{osc}} - 24\omega K_{\mbox{osc}}^{\frac{1}{2}}-8 \omega^2 \geq \delta >0    \mbox{.} 
\end{equation}
From Lemma \ref{lk1}, we obtain
 \begin{equation*} 
    \frac{\mathrm{d} }{\mathrm{d} t}K_{\mbox{osc}}\leq - \delta L_0\left\|k_{ss} \right\|_{2}^{2} \mbox{}
 \end{equation*}
Therefore,  \begin{equation*} 
    \frac{\mathrm{d} }{\mathrm{d} t}\left ( \frac{K_{\mbox{osc}}}{L_0} \right )\leq - \delta \left\|k_{ss} \right\|_{2}^{2} \mbox{}
 \end{equation*}
Hence, using Proposition \ref{psw} twice, we have
$$\frac{\mathrm{d} }{\mathrm{d} t}\left ( \frac{K_{\mbox{osc}}}{L_0} \right ) \leq - \delta \,  \frac{\pi^4}{L{_0}^4} \left ( \frac{K_{\mbox{osc}}}{L_0} \right )   \mbox{}$$ from which the result follows.
\end{proof}

\begin{lem}\label{exksc}
 Let $\alpha :\left [ -1, 1 \right ]\times \left [ 0,  T\right )\to \mathbb{R}^{2}$ be a length-constrained curve diffusion flow with generalised Neumann boundary conditions \eqref{E:NBC1} and \eqref{E:NBC2}. Then there exist constants $C >0$ and $\delta>0$ such that 
  \begin{equation}\label{r3}
  \int_{\alpha } k_s^2\, ds 
  \leq C\, e^{- \frac{\delta \pi^4}{4\, L_0^4} t }\,  , ~~~~~~~~~~~~~\textit{for all}~~~t\geq0,\end{equation}
 that is $\int _{\alpha }k_{s}^{2}\,ds$ decays exponentially to zero.
\end{lem}
\begin{proof}Integrating by parts and using the property at the boundaries, Lemma \ref{T:BCs}, the H\"{o}lder inequality, Corollary \ref{K}, and from Proposition \ref{bd}, that $ \int_{\alpha } k_{ss}^2\, ds\leq\left ( C_1+C_{0}^{9}\, t \right )$,  where $C_1:=\int _{\alpha }k_{ss}^{2}\,ds\bigg{|}_{t=0}$,  we obtain
\begin{align*}
  \int_{\alpha }  k_s^2\,ds 
  = - \int_{\alpha } \left( k - \overline{k} \right) k_{ss}\, ds 
 &\leq \left[ \int_{\alpha } \left( k - \overline{k} \right)^2 ds \right]^{\frac{1}{2}} \left( \int_{\alpha } k_{ss}^2\, ds \right)^{\frac{1}{2}}
 \\& \leq \left ( C_1+C_{0}^{9}\, t \right )^{\frac{1}{2}}\,\delta _{1}^{\frac{1}{2}} \,  e^{- \frac{\delta \pi^4}{2\, L_0^4} t }\leq C\,e^{- \frac{\delta \pi^4}{4\, L_0^4} t } \mbox{,}
  \end{align*}  
which completes the proof.
\end{proof}
\begin{cor} \label{ck5}
Under the flow  \eqref{eq113}, for all $\ell \in \mathbb{N}$, there exist corresponding $\delta_\ell >0$ and $\tilde C_\ell >0$ such that
  \begin{equation*}
 \int _{\alpha }k_{s^\ell}^2\,ds\leq \tilde C_\ell  \, e^{-\delta_\ell t} \mbox{,}~~~~~~~~~~~\textit{for all}~~~t\geq0.
  \end{equation*} 
\end{cor}
\begin{proof}The proof follows, as for Corollary \ref{cor3.14}.
%This is proved by a standard induction argument.  The base case $\ell=1$ is as in Lemma \ref{exksc} \eqref{r3}.  So assume that
 %$$\int _{\alpha }k_{s^\ell}^{2}\,ds\leq \tilde C_\ell e^{-\delta_\ell t},$$ for some $\tilde{C_\ell}>0$.  Integrating by parts, and using the property at the boundaries, Lemma \ref{T:BCs}, we estimate the following by using the H\"{o}lder inequality:
%$$\int _{\alpha }k_{s^{\ell+1}}^2\,ds=-\int _{\alpha }k_{s^{\ell+2}}\,k_{s^{\ell}}\,ds\leq \left ( \int _{\alpha }k_{s^{\ell+2}}^2\,ds \right )^{\frac{1}{2}}\left ( \int _{\alpha }k_{s^{\ell}}^2\,ds \right )^{\frac{1}{2}}.$$
% Using the inductive assumption and  Proposition \ref{bd}, yields $$\int _{\alpha }k_{s^{\ell+1}}^2\,ds\leq C_{\ell+2}\,  \tilde{C_\ell}^\frac{1}{2}e^{-\frac{\delta_\ell}{2} t}.$$
 %Taking $\tilde C_{\ell+1} = C_{\ell+2} \tilde C_\ell^{\frac{1}{2}}$ and $\delta_{\ell+1} = \frac{\delta_\ell}{2}$ completes the proof.     
\end{proof}
\begin{proof}[Completion of the proof of Theorem~{\upshape\ref{T:main2}}]
The argument is similar to that in the previous section, but there is no rescaling here and as deduced earlier, $T=\infty$.  Since the curvature decays pointwise to its average, the limiting curve is the unique circular arc satisfying the boundary conditions and with length $L_0$. Smooth exponential convergence of the embedding map $\alpha\left( \cdot, t\right)$ to that of the circular arc in the $C^\infty$-topology now follows similarly as in \cite{GM24}, for example.  
\end{proof}

% ------------------------------------------------------------------------

\section{Length-constrained curve diffusion flow \texorpdfstring{$F= -k_{ss} +\lambda_2(t)\,k$}{lambda(t) } } \label{S:lfff}  
In this section, we consider
\begin{equation}\label{eq09}
\frac{\partial \alpha }{\partial t}=\bigg(  k_{ss}-\lambda_2(t) \,k\bigg)\,\nu \mbox{.}
\end{equation} 
Taking
\begin{equation}\label{h18}
\lambda_2\left ( t \right ):=\displaystyle{\frac{-\int _{\alpha }k_{s}^{2}\,ds}{\int _{\alpha }k^2\,ds} }
\end{equation}
ensures $L[\alpha_t] = L[\alpha_0] =: L_0$ for all $t$. Similarly as in the previous section we have a bound on the size of $\lambda_2(t)$.

\begin{lem}\label{c212}For each $n \in \mathbb{N} $
$$\left | \lambda_2 (t)\right |\leq \frac{L_0}{\left ( 2\pi \omega  \right )^{2} } \left( \int_{\alpha } k^2 \, ds \right)^{1 - \frac{1}{n}} \left( \int_{\alpha } k_{s^n}^2 \, ds \right)^{\frac{1}{n}}.$$ 
\end{lem}
\begin{proof}
From \eqref{h18}, we have 
$$\left | \lambda_2 (t)\right |= \displaystyle{\frac{\int _{\alpha }k_{s}^{2}\,ds}{\int _{\alpha }k^2\,ds}}$$
by using  H\"{o}lder inequality, we obtain $$\left | \lambda_2 (t)\right |\leq \frac{L_0}{\left ( 2\pi \omega  \right )^{2} }  \int_{\alpha } k_{s}^2 \, ds .$$
The proof now proceeds similarly to that of Lemma \ref{c2}.
\end{proof}

\begin{lem} \label{h34}
Under the flow \eqref{eq09}, we have 
$$\displaystyle{\frac{\mathrm{d} }{\mathrm{d} t}A\left [ \alpha_t \right ]=-2\pi\,\omega\,\lambda_2 (t)}. $$
\end{lem}

In view of \eqref{h18} we observe that the enclosed area is nondecreasing under this flow.

\begin{lem} \label{T:ksL23}
Under the flow \eqref{eq09},
 \begin{multline*}
\frac{\mathrm{d} }{\mathrm{d} t}  K_{\mbox{osc}}=-2L_0\int _{\alpha }k_{ss}^{2}\,ds+3L_0\int _{\alpha }k^{2}k_{s}^{2}\,ds-\bar{k}^{2}L_0\int _{\alpha }k_{s}^{2}\,ds  -2\lambda_2 \left ( t \right )L_0\int _{\alpha }k_{s}^{2}\,ds\\+\lambda_2 \left ( t \right )L_0\int _{\alpha }k^4\,ds -\lambda_2 \left ( t \right )\bar{k}^{2}L_0\int _{\alpha }k^2\,ds
\end{multline*} 
\end{lem}
\begin{proof}
By substituting $F= -\,k_{ss}+\lambda_2(t)\,k$ into Lemma \ref{eq of general F} (vi) and  integrating by parts 
\begin{multline*}
 -\int _{\alpha }k\left ( k^{2}-\bar{k}^{2} \right )k_{ss}\,ds=\int _{\alpha }\left ( k^{2}-\bar{k}^{2} \right )k_{s}^{2}\,ds+2\int _{\alpha }k^{2}k_{s}^{2}\,ds \\
 =-\bar{k}^{2}\int _{\alpha }k_{s}^{2}\,ds+3\int _{\alpha }k^{2}k_{s}^{2}\,ds\,\mbox{,}
\end{multline*} we complete the proof.
\end{proof}
\begin{cor}\label{cokosc}
Under the flow \eqref{eq09},
 \begin{align*}%\label{kosc3}
\frac{\mathrm{d} }{\mathrm{d} t}  K_{\mbox{osc}}  &=\lambda_2 \left ( t \right )L_0\int _{\alpha }\left ( k-\bar{k} \right )^{4}\,ds+5\bar{k}^{2}\lambda_2 \left ( t \right )L_0\int _{\alpha }\left ( k-\bar{k} \right )^{2}\,ds+2\bar{k}^{2}L_0\int _{\alpha }k_{s}^2\,ds\nonumber\\&\qquad+4\bar{k}\lambda_2 \left ( t \right )L_0\int _{\alpha }\left ( k-\bar{k} \right )^{3}\,ds+3L_0\int _{\alpha }\left ( k-\bar{k} \right )^{2}k_{s}^2\,ds \nonumber\\&\qquad+6\bar{k}L_0\int _{\alpha }\left ( k-\bar{k} \right )k_{s}^{2}\,ds-2L_0\int _{\alpha }k_{ss}^{2}\,ds-2\lambda_2 \left ( t \right )L_0\int _{\alpha }k_{s}^{2}\,ds\,\mbox{.}
\end{align*}   
\end{cor}
\begin{proof}
By writing $k=(k-\bar{k})+\bar{k}$ into Lemma \ref{T:ksL23}, the result follows.\end{proof}

We have the following main result for this section.
 \begin{theorem} \label{T:main3}
 Suppose the supporting cone has $\omega < \frac{1}{2}$. Let $\alpha_0: \left[ -1, 1\right]\to \mathbb{R}^{2}$ be compatible with \eqref{E:NBC1} and \eqref{E:NBC2} be a given initial curve of length $L_0$ sufficiently far from the cone tip.  Suppose further that
 $\alpha_0$ has $K_{\mbox{osc}}$ less than  $x>0$, where $x^{\frac{1}{2}}$ is the smallest solution of 
\begin{equation}\label{s2K_{OSC}}
 \left( 2 - 8 \omega^2 \right)-\frac{2}{\omega\,\pi^2}\sqrt{\frac{2}{\pi }} x^{\frac{3}{2}}-24\omega x^{\frac{1}{2}}-\left ( \frac{1}{2\left ( \omega \pi  \right )^{2}}+\frac{ 6}{\pi } \right ) x=0   \mbox{.} 
\end{equation}
 Then the length-constrained curve diffusion flow \eqref{eq09} with generalised Neumann boundary conditions \eqref{E:NBC1} and \eqref{E:NBC2} and initial data $\alpha_0$ has a unique solution $\alpha\left( \cdot, t\right)$ on $\left[ 0, \infty\right)$.  The solution curves converge smoothly and exponentially in the $C^\infty$-topology to the circular arc $\alpha_\infty$ of length $L_0$ centred at the cone tip.\\
\end{theorem}
\begin{remark}
\begin{enumerate}
\item Comparing the coefficients in \eqref{s2K_{OSC}} with those in \eqref{s1K_{OSC}} we observe the smallness requirement on $K_{\mbox{osc}}$ here is more strict.

\item As in the previous section, the smallness condition on $K_{\mbox{osc}}$ may be written down explicitly using the Cardano formula, while the sufficient distance of the boundary points of the initial curve to the cone tip can again be quantified similarly.

\end{enumerate}
\end{remark}

While the proof of Proposition \ref{bd} is again similar in this case, given the form of the flow speed \eqref{eq09}, an additional argument is needed to handle some lower order terms.  We provide this next for completeness.

\begin{proof}[Proof of Proposition~{\upshape\ref{bd}}.]
Using Lemma \ref{T:evlneqns2} (v), where instead of $\lambda$ we have in this case $\lambda_2(t)$ 
\begin{multline}\label{b332}
\frac{\mathrm{d} }{\mathrm{d} t}\frac{1}{2}\int _{\alpha }k_{s^{\ell}}^{2}\,ds
=-\int _{\alpha }k_{s^{\ell+2}}^{2}\,ds-\lambda_2 \left ( t \right ) \int _{\alpha }k_{s^{\ell+1}}^{2}\,ds+\lambda_2 \left ( t \right ) \int _{\alpha }P_{4}^{2\ell}(k)\,ds +\int _{\alpha }P_{4}^{2\ell+2}(k)\,ds\\
\leq -(1-\varepsilon)\int _{\alpha }k_{s^{\ell+2}}^{2}\,ds-\lambda_2 \left ( t \right ) \int _{\alpha }k_{s^{\ell+1}}^{2}\,ds+\lambda_2 \left ( t \right ) \int _{\alpha }P_{4}^{2\ell}(k)\,ds+c\left ( \varepsilon  \right )\left (\int _{\alpha }k^{2}\,ds  \right )^{2\ell+5}.
\end{multline} 
For the term $-\lambda_2 \left ( t \right )\int _{\alpha }k_{s^{\ell+1}}^{2}\,ds$, by using [Corollary 2.4] as in \cite{MWY}, we have for each $n\in \mathbb{N}$
$$ \frac{\int _{\alpha }k_{s}^{2}\,ds}{\left ( 2\pi \omega  \right )^{2}}\leq \frac{1}{\left ( 2\pi \omega  \right )^{2} }\left ( \int _{\alpha }k^{2}\,ds \right )^{1-\frac{1}{n}} \left (  \int _{\alpha }k_{s^{n}}^{2}\,ds \right )^{\frac{1}{n}}$$
therefore, $$\int _{\alpha }k_{s}^{2}\,ds\leq \left ( \int _{\alpha }k^{2}\,ds \right )^{1-\frac{1}{n}} \left (  \int _{\alpha }k_{s^{n}}^{2}\,ds \right )^{\frac{1}{n}}$$
thus, $$\frac{\int _{\alpha }k_{s}^{2}\,ds}{\int _{\alpha }k^{2}\,ds}\leq \left ( \int _{\alpha }k^{2}\,ds \right )^{-\frac{1}{n}} \left (  \int _{\alpha }k_{s^{n}}^{2}\,ds \right )^{\frac{1}{n}}$$ 
in particular for $n=\ell+1$, we have $$\frac{\int _{\alpha }k_{s}^{2}\,ds}{\int _{\alpha }k^{2}\,ds}\,\int _{\alpha }k_{s^{\ell+1}}^{2}\,ds\leq \left ( \int _{\alpha }k^{2}\,ds \right )^{-\frac{1}{\ell+1}} \left (  \int _{\alpha }k_{s^{\ell+1}}^{2}\,ds \right )^{1+\frac{1}{\ell+1}}.$$\\
By using Proposition \ref{psw} and Young's inequality, we obtain
\begin{align*}
\frac{\int _{\alpha }k_{s}^{2}\,ds}{\int _{\alpha }k^{2}\,ds}\,\int _{\alpha }k_{s^{\ell+1}}^{2}\,ds &\leq \varepsilon \int _{\alpha }k_{s^{\ell+2}}^{2}\,ds-c\varepsilon ^{\left ( \frac{1}{\ell+1}\right )}   \int _{\alpha }k^{2}\,ds. 
\end{align*}
For the term $\lambda_2 \left ( t \right ) \int _{\alpha }P_{4}^{2\ell}(k)\,ds$, we proceed similarly as in \cite{MWY}, and make the necessary adjustments
$$\lambda_2(t) \int_{\alpha} P_4^{2\ell}(k) \, ds \leq 2\varepsilon \int_{\alpha}  k_{s^{\ell+2}} ^2 \, ds + c\left ( \varepsilon  \right ) \left( \int_{\alpha } k^2\, ds \right)^{\frac{3\ell + 5}{\ell + 1}} 
+ c\left ( \varepsilon  \right ) \left( \int_{\alpha } k^2\, ds \right)^{\frac{4\ell + 9}{3}}.$$

Substituting these estimates into \eqref{b332}, we have for small $\varepsilon$
\begin{equation}\label{erk}
\frac{\mathrm{d} }{\mathrm{d} t}\frac{1}{2}\int _{\alpha }k_{s^{\ell}}^{2}\,ds\leq - \int _{\alpha }k_{s^{\ell+2}}^{2}\,ds+ C   \left ( \int _{\alpha }k^{2}\,ds \right )^{2\ell+5} + c \left( \int_{\alpha } k^2\, ds \right)^{\frac{3\ell + 5}{\ell + 1}} 
+ c \left( \int_{\alpha } k^2\, ds \right)^{\frac{4\ell + 9}{3}}.
\end{equation}
Discarding the first term on the right-hand side of \eqref{erk}, and in view of the boundedness of $\int_\alpha k^2\, ds$, the result follows by integration over the finite time interval. %The bound on $\int_\alpha k_{s^\ell}^2\, ds$ then follows using Proposition \ref{psw} twice in reverse.
\end{proof}

\begin{cor}\label{kosk5}
Under the flow \eqref{eq09},
 \begin{multline*}
\frac{\mathrm{d} }{\mathrm{d} t}  K_{\mbox{osc}}  \leq 4\bar{k}L_0\lambda_2 \left ( t \right )\int _{\alpha }\left ( k-\bar{k} \right )^{3}\,ds+3L_0\int _{\alpha }\left ( k-\bar{k} \right )^{2}k_{s}^{2}\,ds +6\bar{k}L_0\int _{\alpha }\left ( k-\bar{k} \right )k_{s}^{2}\,ds\\
-2L_0\int _{\alpha }k_{ss}^{2}\,ds -2L_0\lambda_2 \left ( t \right )\int _{\alpha }k_{s}^{2}\,ds+2L_0\bar{k}^{2}\int _{\alpha }k_{s}^{2}\,ds
\end{multline*}  
\end{cor}
\begin{proof}
Discarding the first and the second  non-positive terms that are not useful in Corollary \ref{cokosc}, completes the proof. \end{proof}
\begin{cor}\label{eqko}
Under the flow \eqref{eq09},
\begin{equation*}
 \frac{\mathrm{d} }{\mathrm{d} t}  K_{\mbox{osc}} \leq  - L_0\left[ 2-8\omega ^2-\frac{2}{\omega\,\pi^2}\sqrt{\frac{2}{\pi }} K_{\mbox{osc}}^{\frac{3}{2}}-24\,\omega K_{\mbox{osc}}^{\frac{1}{2}}-\left ( \frac{1}{2\omega^2 \pi^2}+\frac{ 6}{\pi } \right ) K_{\mbox{osc}}\right]\int _{\alpha }k_{ss}^{2}\,ds \mbox{.}
\end{equation*}
Consequently, if $\omega < \frac{1}{2}$ and $K_{\mbox{osc}}\left[ \alpha_0\right]$ satisfies the smallness condition of Theorem \ref{T:main3}, we have that $K_{\mbox{osc}}$ does not increase under the flow.
\end{cor}
\begin{proof}
 As in the proof of Lemma \ref{lk1},   
 $$3L_0\int _{\alpha }\left ( k-\bar{k} \right )^{2}k_{s}^{2}\,ds\leq \frac{6L_0}{\pi }K_{\mbox{osc}}\left\|k_{ss} \right\|_{2}^{2},$$
 $$6\bar{k}L_0\int _{\alpha }\left ( k-\bar{k} \right )k_{s}^{2}\,ds\leq 24\omega L_0 \sqrt{K_{\mbox{osc}}} \left\|k_{ss} \right\|_{2 }^{2}\mbox{.}$$ and\begin{align*}
 2L_0\bar{k}^{2}\left\|k_{s} \right\|_{2}^{2}&\leq 8\omega ^{2}L_0\left\|k_{ss} \right\|_{2}^{2}.
\end{align*}

We estimate the term $4L_0 \,\bar{k}\,\lambda(t)  \int _{\alpha } (k - \overline{k})^3 \, ds $ as follows:
\begin{multline}\label{ek51}
 \int_{\alpha} k_s^2\, ds = - \int_{\alpha} k\, k_{ss} \, ds 
= - \int_{\alpha} (k - \bar{k})\, k_{ss} \, ds \\
\leq \left( \int_{\alpha} (k - \bar{k})^2 \, ds \right)^{\frac{1}{2}} \left( \int_{\alpha} k_{ss}^2 \, ds \right)^{\frac{1}{2}}
= \frac{K_{\mbox{osc}}^{\frac{1}{2}}}{L_{0}^{\frac{1}{2}}}\,\left\| k_{ss}\right\|_2\mbox{,}
\end{multline} 
and 
\begin{multline}\label{ek52}
\int_{\alpha} (k - \bar{k})^3 \, ds 
\leq \|k - \bar{k}\|_{\infty} \int_{\alpha} (k - \bar{k})^2 \, ds 
 \leq \sqrt{\frac{2L_0}{\pi }}\left\| k_s\right\|_{2}\,\int _{\alpha }\left ( k-\bar{k} \right )^2\,ds\\
 \leq \frac{L_0}{\pi }\sqrt{\frac{2L_0}{\pi }}\left\| k_{ss}\right\|_{2}\,\int _{\alpha }\left ( k-\bar{k} \right )^2\,ds
 =\frac{1}{\pi }\sqrt{\frac{2L_0}{\pi }}\left\| k_{ss}\right\|_{2}K_{\mbox{osc}} \mbox{;}
 \end{multline}
therefore,
\begin{align*}
4L_0 \,\bar{k}\, \lambda_2(t) \int_{\alpha} (k - \bar{k})^3 \, ds 
&\leq 4L_0\, \bar{k}\,\left | \lambda_2 \left ( t \right )\right |\,\int _{\alpha }\left ( k-\bar{k} \right )^{3}\,ds\,\mbox{.}
\end{align*}
By using Lemma \ref{c212}, when $n=1$, we have 
\begin{multline*}
 4L_0\, \bar{k}\,\left | \lambda_2 \left ( t \right )\right |\,\int _{\alpha }\left ( k-\bar{k} \right )^{3}\,ds
 \leq \frac{4L_{0}^{2}\bar{k}}{\left ( 2\omega\, \pi  \right )^{2}}\int _{\alpha }k_{s}^2\,ds\,\int _{\alpha }\left ( k-\bar{k} \right )^{3}\,ds\\
 \leq \frac{2L_{0}} {\omega\, \pi  }\int _{\alpha }k_{s}^2\,ds\,\int _{\alpha }\left ( k-\bar{k} \right )^{3}\,ds,
\end{multline*}
thus, by using \eqref{ek51} and \eqref{ek52}, we obtain 
\begin{equation*}
4L_0 \,\bar{k}\, \lambda_2(t) \int_{\alpha} (k - \bar{k})^3 \, ds 
\leq  \frac{2L_0}{\pi^2 \omega }\sqrt{\frac{2}{\pi }} K_{\mbox{osc}}^{\frac{3}{2}} \|k_{ss}\|_2^2.
\end{equation*}
 For $-2L_0\lambda_2(t)\left\|k_{s} \right\|_{2}^{2}$, we have \begin{multline*}
 -2L_0\lambda_2(t)\left\|k_{s} \right\|_{2}^{2}\leq \frac{2L_0\int _{\alpha }k_{s}^{2}\,ds}{\int _{\alpha }k^{2}\,ds}\,\int _{\alpha }k_{s}^{2}\,ds
 \leq \frac{L_{0}^{2}}{2\left ( \omega\, \pi  \right )^2}\,\int _{\alpha }k_{s}^{2}\,ds\,\int _{\alpha }k_{s}^{2}\,ds\\
 \leq \frac{L_{0}^{2}}{2\left ( \omega\, \pi  \right )^2}\,  \int_{\alpha} (k - \bar{k})^2 \, ds  \int_{\alpha} k_{ss}^2 \, ds
 \leq \frac{L_{0}}{2\left ( \omega\, \pi  \right )^2}\,K_{\mbox{osc}}\left\| k_{ss}\right\|_{2}^{2}\mbox{.}
\end{multline*}
By inserting these estimates into  Corollary  \ref{kosk5}, we complete the proof.
\end{proof}

The proof of Theorem \ref{T:main3} is now completed using the same approach as in the previous section.

\backmatter

%\bmhead{Supplementary information}

%If your article has accompanying supplementary file/s please state so here. 

%Authors reporting data from electrophoretic gels and blots should supply the full unprocessed scans for key as part of their Supplementary information. This may be requested by the editorial team/s if it is missing.

%Please refer to Journal-level guidance for any specific requirements.

\bmhead{Acknowledgements}
The research of the first author was supported by a postgraduate scholarship from the Department of Mathematics, College of Science, Imam Abdulrahman Bin Faisal University, P. O. Box 1982, Dammam, Saudi Arabia.  Part of this research was completed while the second author was visiting the University of Science and Technology, China under a Chinese Academy of Sciences Presidents' International Fellowship Initiative visiting fellowship, grant number 2024PVA0042.  Part of this research was completed while the second author was supported by Discovery Project DP250103952 of the Australian Research Council.  The authors are grateful for this support.  The authors are also grateful to the referees of an earlier version of this article for their feedback that has led to improvements.

\section*{Statements and Declarations}

%Some journals require declarations to be submitted in a standardised format. Please check the Instructions for Authors of the journal to which you are submitting to see if you need to complete this section. If yes, your manuscript must contain the following sections under the heading `Declarations':

\begin{itemize}
\item Funding

The research of the first author was supported by a postgraduate scholarship from the Department of Mathematics, College of Science, Imam Abdulrahman Bin Faisal University, P. O. Box 1982, Dammam, Saudi Arabia.  Part of this research was completed while the second author was visiting the University of Science and Technology, China under a Chinese Academy of Sciences Presidents' International Fellowship Initiative visiting fellowship, grant number 2024PVA0042.  Part of this research was completed while the second author was supported by Discovery Project DP250103952 of the Australian Research Council. 

\item Competing interests 

The authors declare no competing interests.

\item Ethics approval and consent to participate

Not applicable.

\item Consent for publication

The authors consent to the publication of this work.

\item Data availability 

Not applicable.

\item Materials availability

Not applicable.

\item Code availability 

Not applicable.

\item Author contribution

Each author contributed equally to the production of this work.
\end{itemize}

%\noindent
%If any of the sections are not relevant to your manuscript, please include the heading and write `Not applicable' for that section. 

%\bigskip

%\bigskip\noindent
%Springer journals and proceedings: \url{https://www.springer.com/gp/editorial-policies}

%\bigskip\noindent
%Nature Portfolio journals: \url{https://www.nature.com/nature-research/editorial-policies}

%\bigskip\noindent
%\textit{Scientific Reports}: \url{https://www.nature.com/srep/journal-policies/editorial-policies}

%\bigskip\noindent
%BMC journals: \url{https://www.biomedcentral.com/getpublished/editorial-policies}
%\end{flushleft}

\begin{appendices}




\end{appendices}



\end{document}